\documentclass[a4paper,12pt,oneside]{article}
\usepackage[utf8]{inputenc}   
\usepackage[english]{babel}
\usepackage{verbatim}
\usepackage{amssymb,amsfonts,amsmath,amsthm,qsymbols}
\usepackage{cancel,enumerate}

\usepackage[scale=0.80]{geometry}
\usepackage{color}

\newcounter{a}
\newenvironment{cond}{\par \refstepcounter{a} 
 {\upshape \textbf{Cond~\thea}:}~}{\par}
\def \C {\textbf{Cond}}

\usepackage{graphicx}

\newcommand{\app}[5]{\begin{array}{rccl} #1:&#2&\longrightarrow&#3\\ &#4&\longmapsto&#5\end{array}}  

\newcommand{\indic}[1]{\mathbf{1}_{#1}}

\newcommand{\reals}{\mathbb{R}}

\newcommand{\card}[1]{|#1|}

\newcommand{\comp}[1]{#1^c}
\newcommand{\cyl}[1]{\mathbb{Z} / #1 \mathbb{Z}}

\newcommand{\module}[1]{\left| #1 \right|}

\newcommand{\RN}{Radon-Nikodym}

\theoremstyle{plain}
\newtheorem{theoreme}{Theorem}[section]
\newtheorem{lemme}[theoreme]{Lemma}
\newtheorem{proposition}[theoreme]{Proposition}

\newtheorem{remarque}[theoreme]{Remark}
\theoremstyle{definition}
\newtheorem{definition}[theoreme]{Definition}
\newtheorem{exemple}[theoreme]{Example}
\newtheorem*{notation}{Notation}

\renewcommand{\d}{\mathrm{d}} 
\newcommand{\prob}[1]{P \left( #1 \right)}

\newcommand{\borel}[1]{\mathcal{B}(#1)}

\newcommand{\HZ}{\text{HZ}}
\newcommand{\CHZ}{\text{CHZ}}

\newcommand{\defi}[1]{\emph{#1}}

\newcommand{\truc}[1]{\overline{#1}}

\newcommand{\PCA}[1]{\mathbf{#1}}

\title{\Large \bf Probabilistic cellular automata with general alphabets letting a Markov chain invariant.}
\author{\bf Jérôme Casse \\ Univ. Bordeaux, LaBRI \\ UMR 5800\\ F-33400 Talence, France}
\date{}

\begin{document}
\maketitle
\begin{abstract}
This paper is devoted to probabilistic cellular automata (PCA) on $\mathbb{N}$, $\mathbb{Z}$ or $\mathbb{Z}/n\mathbb{Z}$, depending of two neighbors, with a general alphabet $E$ (finite or infinite, discrete or not). We study the following question: under which conditions does a PCA possess a Markov chain as invariant distribution? Previous results in the literature give some conditions on the transition matrix (for positive rate PCA) when the alphabet $E$ is finite. Here we obtain conditions on the transition kernel of PCA with a general alphabet $E$. In particular, we show that the existence of an invariant Markov chain is equivalent to the existence of a solution to a cubic integral equation.\par
One of the difficulties to pass from a finite alphabet to a general alphabet comes from some problems of measurability, and a large part of this work is devoted to clarify these issues.
\end{abstract}
    
\section{Introduction}
\subsubsection*{CA and PCA with finite alphabet}
\defi{Cellular automata} (CA), as described by Hedlund~\cite{Hedlund69}, are discrete local dynamical systems on a space $E^{\mathbb{L}}$ where $E = \{0,\dots,\kappa\}$ is a finite alphabet, the set of states of cells, and $\mathbb{L}$ is a discrete lattice. Formally, a cellular automaton $A$ is a tuple $(\mathbb{L},E,N,f)$ where
\begin{itemize}
  \setlength{\itemsep}{1pt}
  \setlength{\parskip}{0pt}
  \setlength{\parsep}{0pt}
\item $\mathbb{L}$ is a lattice, called set of cells. In this paper, $\mathbb{L}$ is $\mathbb{N}$, $\mathbb{Z}$ or $\cyl{n}$.
\item $N$ is the neighborhood function: for $i \in \mathbb{L}$, $N(i) = (i+l:~l\in L)$ where $L \subset \mathbb{L}$ is finite. Each neighborhood has cardinality $\card{N} = \card{L}$. In the paper, $N(i) = (i,i+1)$ when the lattice is $\mathbb{N}$ or $\mathbb{Z}$ and $N(i) = (i,i+1 \mod n)$ when the lattice is $\cyl{n}$.
\item $f$ is the local rule. It is a function $f : E^{\card{N}} \to E$. 
\end{itemize}\par

The CA $\PCA{A} = (\mathbb{L},E,N,f)$ defines a global function $F : E^{\mathbb{L}} \mapsto E^{\mathbb{L}}$ on the set of configurations $E^{\mathbb{L}}$. For any configuration $S_0 = (S_0(i):~i \in \mathbb{L})$, the image $S_1 = F(S_0) $ of $S_0$ by $F$ is defined by, for any $j \in \mathbb{L}$, 
\begin{displaymath}
S_1(j) = f \big( \left( S_0(i):~i \in N(j) \right) \big).
\end{displaymath}
In words, the state of all cells are updated simultaneously and the state $S_1(j)$ of the cell $j$ at time $1$ depends only of the states $\left(S_0\left(i\right):~i \in N(j)\right)$ of its neighborhood at time $0$. Hence, the dynamic is the following: starting from an initial configuration $S_{t_0} \in E^{\mathbb{L}}$ at time $t_0$, the successive states of the system are $\left(S_t:~t \geq t_0\right)$ where $S_{t+1} = F\left(S_t\right)$. The sequence of configurations 
\begin{displaymath}
S = (S_t = (S_t(i):~i\in \mathbb{L}),~t \geq t_0)
\end{displaymath}
is called the space-time diagram of $\PCA{A}$.\par

\begin{notation}
In the following, the state $S_t(i)$ of the cell $i$ at time $t$ will be denoted $S(i,t)$.
\end{notation}

\defi{Probabilistic cellular automata} (PCA) with finite alphabets are generalizations of CA in which the states  $\left(S(i,t):~i \in \mathbb{L},~t \geq t_0\right)$ are random variables (r.v.) defined on a common probability space $(\Omega,\mathcal{A},\mathbb{P})$, each of the r.v. $S(i,t)$ taking a.s. its value in $E$. Seen as random process, $S$ is equipped with the $\sigma$-field generated by the cylinders. To define PCA, the local rule $f$ is replaced by a transition matrix $T$ (of size $E^{\card{N}} \times E$) which gives the distributions of the state of a cell at time $t+1$ conditionally on those of its neighborhood at time $t$:
\begin{displaymath}
\prob{S(j,t+1) = b~|~\left(S(i,t) = a_i:~i \in N(j)\right)} = T\left(\left(a_i:~i \in N(j)\right);b\right).
\end{displaymath}
Conditionally on $S_t$, the states $(S(j,t+1):~j \in \mathbb{L})$ are independent (see Eq~(\ref{eq:ind})).\par

The transition matrix $T$ is then an array of non negative numbers satisfying, for any $\left(a_1,\dots,a_{\card{N}}\right) \in E^{\card{N}}$, $\displaystyle \sum_{b \in E} T\left(\left(a_1,\dots,a_{\card{N}}\right);b\right) = 1$.

Formally, a PCA $\PCA{A}$ with a finite alphabet $E$ is an operator $\mathcal{F} : \mathcal{M}\left(E^{\mathbb{L}}\right) \mapsto \mathcal{M}\left(E^{\mathbb{L}}\right)$  on the set of probability distributions $\mathcal{M}\left(E^{\mathbb{L}} \right)$ on the set of configurations. If $S_0$ has distribution $\mu_0$, then $S_1$ has distribution $\mu_1 = \mathcal{F}\left(\mu_0\right)$. We can, also, define $\mu_1$ directly from $\mu_0$ and $T$, using Kolmogorov extension theorem ($\mu_1$ is characterized by its finite-dimensional distributions), by: for any finite subset $C \subset \mathbb{L}$ and for any $\left(b_j:~j \in C \right) \in E^C$, 
\begin{equation} \label{eq:ind}
\mu_1 \big( \left(b_j:~j \in C\right) \big) = \sum_{\left(a_i\right)_{i \in N(C)} \in E^{N(C)}} \mu_0 \big( \left(a_i:~i \in N(C)\right) \big) \prod_{j \in C} T \left( \left(a_i:~i \in N(j)\right);b_j \right)
\end{equation}
where $\displaystyle N(C) = \bigcup_{j \in C} N(j)$. A measure $\mu \in \mathcal{M} \left( E^{\mathbb{L}} \right)$ is said to be \emph{invariant} by $\PCA{A}$ if $\mathcal{F}(\mu) = \mu$.\par

The simplest case of PCA is the two colors case $E=\{0,1\}$ on $\mathbb{Z}$ with neighborhood $N(i)=(i,i+1)$. They have been deeply studied and lots of results about them are known, see Toom~\cite{Toom90}. For example, Belyaev~\cite{Bel69} characterized the set of PCA possessing as invariant distribution a Markov chain indexed by $\mathbb{Z}$. Nevertheless, there are still interesting open problems about them: for instance, the question whether all positive rate PCA (i.e., for any $a,b,c \in \{0,1\}$, $T(a,b;c)>0$) are ergodic or not is still open.\par

So far, it has been observed in different frameworks that explicit calculus of the invariant distribution of PCA can be done only if the transition matrix satisfies some algebraic equations (that forms a manifold in terms of the $\left(T(a,b;c):~a,b,c \in E\right)$). In Belyaev~\cite{Bel69} this is shown for PCA with 2-letter alphabet whose invariant distributions are Markov chains or product measures. In Dai-Pra~\cite{DaiPra02}, this is done for PCA on $\mathbb{Z}^d$ with a 2-letter alphabet and whose invariant distributions are Gibbs measures. And, in Casse and Marckert~\cite{CasMar14}, the same phenomenon is observed for PCA on $\mathbb{Z}$ or $\cyl{n}$ with a finite alphabet letting a Markov chain invariant. Hence, literature focuses on characterizing PCA having simple invariant measures: product measures and Markov chains for $\card{N}=2$ and Gibbs measures for PCA on $\mathbb{Z}^d$. In addition to~\cite{Bel69}, the study of PCA on $\mathbb{Z}$ admitting an invariant product measure have been done by Mairesse and Marcovici~\cite{MaiMar12} (in a finite alphabet case). For PCA letting a Markov chain invariant, in addition to~\cite{Bel69} and~\cite{CasMar14}, Bousquet-Mélou~\cite{MBM98} characterizes those on $\cyl{n}$ with a 2-letter alphabet and Toom~\cite{Toom90} gives a sufficient condition for PCA on $\mathbb{Z}$ with a finite alphabet.\par

The most general results are given in~\cite{CasMar14} where it is proved (in Theorem~2.6) that a positive PCA on $\mathbb{Z}$ with two neighbors and a finite alphabet $E=\{0,\dots,\kappa\}$ admits an horizontal zigzag Markov chain (see Definition~\ref{def:HZMC}) as invariant distribution if and only if the two following conditions are satisfied:
\begin{enumerate}
  \setlength{\itemsep}{1pt}
  \setlength{\parskip}{0pt}
  \setlength{\parsep}{0pt}
\item for any $a,b,c \in E$, 
\begin{displaymath}
T(a,b;c) T(a,0;0) T(0,b;0) T(0,0;c) = T(a,b;0) T(a,0;c) T(0,b;c) T(0,0;0) \text{ and }
\end{displaymath}
\item $D^\gamma U^\gamma = U^\gamma D^\gamma$
\end{enumerate}
where $\displaystyle D^\gamma(a;c) = \frac{\displaystyle \sum_{k\in E} \frac{\gamma(k)}{T(a,k;0)} T(a,k;c)}{\displaystyle \sum_{k \in E} \frac{\gamma(k)}{T(a,k;0)}}$ and $\displaystyle U^\gamma(c;b) = \frac{\displaystyle \frac{\gamma(b)}{T(0,b;0)} T(0,b;c)}{\displaystyle \sum_{k\in E} \frac{\gamma(k)}{T(0,k;0)} T(0,k;c)}$ for any $a,b,c \in E$ where $\gamma$ is an eigenvector of an explicit matrix that depends only of $T$. 
This theorem is an extension of Theorem~3 of~\cite{Bel69} valid only for $2$ letters alphabet.\par

Inspired by this recent work, we investigate in this paper the case where the alphabet $E$ is general (finite or infinite, discrete or not). As we have to define probability distributions on $E$, as usual in probability theory, we will assume that $E$ is a Polish space (a separable complete metrizable space) equipped with its Borel set $\borel{E}$. It could be finite or infinite and discrete or not. In the following, when we write ``general alphabet'', we are thinking about a Polish space alphabet.\par

 Let us, first, define formally a PCA with a general alphabet.

\subsubsection*{PCA with general alphabet}
For PCA with general alphabets, transition matrices are replaced by transition kernels: let $F$ and $G$ be two Polish spaces, $K=\left( K(x;Y):~x \in F,~Y \in \borel{G} \right)$ is a \defi{transition kernel} (t.k.) from $F$ to $G$ if
\begin{enumerate}
  \setlength{\itemsep}{1pt}
  \setlength{\parskip}{0pt}
  \setlength{\parsep}{0pt}
\item $x \mapsto K(x;Y)$ is $\borel{F}$-measurable for all $Y \in \borel{G}$,
\item $Y \mapsto K(x;Y)$ is a probability measure on $(G,\borel{G})$ for all $x \in F$.
\end{enumerate}

\begin{definition}[Probabilistic cellular automata with a general alphabet]
Let $E$ be a Polish space, $\mathbb{L}$ a lattice, $N$ a neighborhood function and $T$ a t.k. from $E^{\card{N}}$ to $E$. The PCA $\PCA{A} = (E,\mathbb{L},N,T)$ is the dynamical system on $E^\mathbb{L}$ such that, for all $k \in \mathbb{N}$, for all $j_1,\dots,j_k \in \mathbb{L}$, $C_1,\dots,C_k \in \borel{E}$, $t \in \mathbb{N}$,
\begin{displaymath}
\prob{S(j_1,t+1) \in C_1, \dots, S(j_k,t+1) \in C_k~|~\left(S(i,t):~i \in \mathbb{L}\right)} = \prod_{l=1}^k T\left( \left(S(i,t):~i \in N(j_l) \right) ; C_l\right).
\end{displaymath}
\end{definition}

If $E$ is finite, this definition is similar to the classical definition of PCA. But, now, the alphabet $E$ can be non-discrete and the t.k. can contain a non-atomic part.

\begin{exemple}[Gaussian PCA] \label{ex:GaussianPCA}
We define a family of PCA $(\PCA{G}_{m,\sigma})$ on $\mathbb{N}$ with alphabet $\reals$ and neighborhood $N(i)=(i,i+1)$ depending on two positive parameters $m$ and $\sigma$. The t.k. of $\PCA{G}_{m,\sigma}$ is the following: for all $a,b \in \reals$ and Borel set $C \in \borel{\reals}$,
\begin{displaymath}
T(a,b;C) = \prob{\mathcal{N} \left(\frac{a+b}{m},\sigma^2 \right) \in C}
\end{displaymath}
where $\displaystyle \mathcal{N}\left(\frac{a+b}{m},\sigma^2 \right)$ is the Gaussian random variable with mean $\displaystyle \frac{a+b}{m}$ and variance $\sigma^2$. In Section~\ref{sec:GaussianPCA}, we prove that an invariant measure of this PCA is related to autoregressive processes of order $1$ (AR(1)~processes).\par 
\end{exemple}

PCA with infinite and non-discrete alphabets exist in the literature, even if they are not studied as such to the best knowledge of the author. For example, in Section~\ref{sec:PCAlit}, we will see that the synchronous TASEP on $\reals$ defined by Blank~\cite{Blank12} (it is a discrete time, synchronous, space continuous version of the TASEP studied by Derrida \& al.~\cite{Derrida92}) could be modeled by a PCA on $\mathbb{Z}$ with alphabet $E=\reals$ and neighborhood $N(i) = (i,i+1)$.\par

The \emph{aim of the paper} is to shed some light on the structure of the set of PCA with a general alphabet (finite or infinite, discrete or not) having a Markovian invariant distribution on lattices $\mathbb{N}$, $\mathbb{Z}$ or $\cyl{n}$. In this case, some important complications arise due to measurability issues (compared with the finite case).\par

In continuous probability, it is classical that two distributions having a density are equal, if these densities are equal almost everywhere for the Lebesgue measure. This fact holds in a more general context: if $\mu$ is a $\sigma$-finite measure and $\nu_1$ and $\nu_2$ two measures absolutely continuous with respect to $\mu$, they are equal if their \RN-derivatives with respect to $\mu$ are equal $\mu$-almost everywhere. But, for Markov chains with general alphabets, the equality $\mu$-almost everywhere for transition kernel (i.e. for $\mu$-almost $x$, $M(x;.)=M'(x;.)$) is not a sufficient condition to have equality in distribution for the Markov chains. Indeed, it also depends of the initial probability distribution $\rho$ of these Markov chains. If $\rho$ charges a $\mu$-negligible set on which $M(x;.) \neq M'(x;.)$, then the two Markov chains $\left(\rho,M\right)$ and $\left(\rho,M'\right)$ can be not equal in distribution, and even, live on two different sets.\par

For PCA with any general alphabet, the same complications arise: a unique PCA can have some ``plural behaviors''. Hence, in this paper, each time a PCA $\PCA{A}$ is studied, a $\sigma$-finite measure $\mu$ is specified and, formally, it is on the pair $(\PCA{A},\mu)$ that the conditions and/or results hold.\par

\begin{exemple}[Gaussian PCA except on the diagonal] \label{ex:GaussianPCAd}
We define a family of PCA $\left( \PCA{\tilde{G}}_{m,\sigma} \right)$ on $\mathbb{N}$ with alphabet $\reals$ depending on two positives parameters $m$ and $\sigma$. The t.k. $\tilde{T}$ of $\PCA{\tilde{G}}_{m,\sigma}$ is the same as for $\PCA{G}_{m,\sigma}$ (defined in Example~\ref{ex:GaussianPCA}) except when $a=b$, in this case, for any $C \in \borel{\reals}$, $\tilde{T}(a,a;C) = \delta_a(C)$ where $\delta_a$ is the Dirac measure in $a$.\par

The PCA $\PCA{\tilde{G}}_{m,\sigma}$ will have the same behavior as the Gaussian PCA $\PCA{G}_{m,\sigma}$ if the initial state $S_{t_0}$ does not contain two consecutive cells in the same state, i.e. for any $i$, $S(i,t_0) \neq S(i+1,t_0)$. But, if its initial state is $0^{\mathbb{N}}$, then it will stay in this configuration until the end.
\end{exemple}

Before introducing the set of studied PCA in this article, let define some crucial notion used all along the paper: $\mu$-supported and $\mu$-positive transition kernels.
\begin{notation}
If $\mu$ is a measure on $E$ and $d \in \mathbb{N}$, then $\mu^d$ denotes the product measure of $d$ copies of the measures $\mu$ on $E^d$.
\end{notation}
\begin{definition}[$\mu$-supported and $\mu$-positive transition kernels]
Let $E$ be a Polish space, $\mu$ a $\sigma$-finite measure on $E$ and $d \in \mathbb{N}$. Let $K$ be a transition kernel from $E^d$ to $E$, $K$ is said to be \defi{$\mu$-supported} if for $\mu^d$-almost $\left(x_1,\dots,x_d\right)$, $K\left(x_1,\dots,x_d;.\right) \ll \mu$; if, moreover, for $\mu^d$-almost $\left(x_1,\dots,x_d\right)$, $\mu \ll K\left(x_1,\dots,x_d;.\right)$, then $K$ is said to be \defi{$\mu$-positive}.

For $K$ a $\mu$-supported transition kernel from $E^d$ to $E$, the \emph{$\mu$-density} of $K$ is the $\mu^{d+1}$-measurable function $k$ such that
\begin{displaymath}
\app{k}{E^{d+1}}{\mathbb{R}}{k\left(x_1,\dots,x_d;y\right)}{\displaystyle\frac{\d K\left(x_1,\dots,x_d;.\right)}{\d \mu}(y).}
\end{displaymath}

If, moreover, $K$ is $\mu$-positive, then, for $\mu^{d+1}$-almost $\left(x_1,\dots,x_d,y\right)$, $k\left(x_1,\dots,x_d;y\right)>0$.
\end{definition}

In the following, we will work with $\mu$-supported or $\mu$-positive kernels for $d=1$ (transition kernels of Markov chain) or $d=\card{N}=2$ (transition kernels of PCA).

We will see that such transition kernels permit to work with densities instead of measures. In the following, the \RN-derivative of any measure with respect to $\mu$ will be also shorten in $\mu$-density.\par 

An example of a Lebesgue-supported t.k. is the t.k. $T$ of Gaussian PCA (defined in Example~\ref{ex:GaussianPCA}). This t.k. is even Lebesgue-positive. In the following, we call a $\mu$-supported (resp. $\mu$-positive) PCA a PCA whose t.k. is $\mu$-supported (resp. $\mu$-positive).\par

We will make apparent below (in particular in Section~\ref{sec:PCAfa} and~\ref{sec:GaussianPCA}) that to describe the invariant distribution of a PCA, at least in the case where it admits a Markov chain as invariant distribution, that we have to work under a reference measure $\mu$, which depending on the case can be the Lebesgue measure, a discrete measure, or any $\sigma$-finite measure. An example of that is the PCA $\PCA{\tilde{G}}_{m,\sigma}$ of Example~\ref{ex:GaussianPCAd} for which we will find different invariant distributions according to whether the reference measure is the Lebesgue-measure or $\delta_a$.\par

\begin{remarque}
There exists some transition kernels that are not $\mu$-supported by any $\sigma$-finite measure $\mu$. For example, the t.k. $T$ from $\reals^2$ to $\reals$ defined by, for any $a,b \in \reals$, $C \in \mathcal{B}(\reals)$, $T(a,b;C) = \begin{cases} \delta_{a}(C) & \text{if } a \neq b \\ \displaystyle \int_C \frac{1}{\sqrt{2 \pi}}e^{\frac{-(c-a)^2}{2}} \d c & \text{if } a=b \end{cases}$ is not $\mu$-supported. Indeed, any measure $\mu$ that could support this PCA has necessarily an atom at each $x$ in $\reals$. Then, $\mu$ is not a $\sigma$-finite measure.
\end{remarque}

Studied PCA in this work are the set of $\mu$-supported PCA and its subset of $\mu$-positive PCA. For both sets, we characterize PCA that have an invariant horizontal zigzag Markov chain, as defined now.\par

Let define the horizontal zigzag Markov chains (HZMC) on $\mathbb{N}$. First, the geometrical structure of \defi{horizontal zigzag} is: the $t$th horizontal zigzag on a space-time diagram is 
\begin{displaymath}
\HZ_{\mathbb{N}}(t) = \left\{ \left( \left\lfloor \frac{i}{2} \right\rfloor, t + \frac{1+(-1)^{i+1}}{2} \right), i \in \mathbb{N} \right\}
\end{displaymath}
as illustrated in Figure~\ref{fig:HZMC}.

\begin{figure}
\begin{center}
\includegraphics{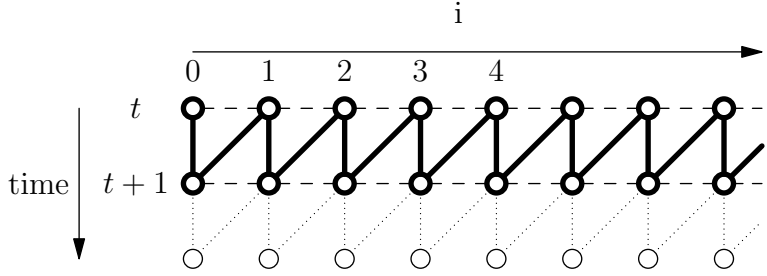}
\caption{In bold, $\HZ_{\mathbb{N}}(t)$, the $t$th horizontal zigzag on $\mathbb{N}$ on a space time diagram.}
\label{fig:HZMC}
\end{center}
\end{figure}

Since $\HZ_{\mathbb{N}}(t)$ is made by two lines corresponding to two successive times, a PCA $\PCA{A}$ on $\mathbb{N}$ can be seen as acting on the configurations of $\HZ_{\mathbb{N}}$. The image of a configuration $\left(S(i,t),S(i,t+1):~i \in \mathbb{N}\right)$ on $\HZ_{\mathbb{N}}(t)$ by the PCA $\PCA{A}$ is $\left(S(i,t+1),S(i,t+2):~i \in \mathbb{N}\right)$ on $\HZ_{\mathbb{N}}(t+1)$. Where the configuration of the second line of $\HZ_{\mathbb{N}}(t)$ becomes the configuration of the first line of $\HZ_{\mathbb{N}}(t+1)$ and the configuration of the second line of $\HZ_{\mathbb{N}}(t+1)$ is the image by $\PCA{A}$ of the second line of $\HZ_{\mathbb{N}}(t)$.\par

\begin{definition} \label{def:HZMC}
An \defi{horizontal zigzag Markov chain} (HZMC) on $\HZ_{\mathbb{N}}(t)$ with general alphabet $E$ is a Markov chain with two t.k. $D$ (for down) and $U$ (for up) from $E$ to $E$ and an initial probability distribution $\rho_0$ on $E$ such that
\begin{enumerate}
  \setlength{\itemsep}{1pt}
  \setlength{\parskip}{0pt}
  \setlength{\parsep}{0pt}
\item the distribution of state $S(0,t)$ is $\rho_0$,
\item the distribution of state $S(i,t+1)$ knowing $S(i,t)=x_i$ is $D(x_i;.)$ and
\item the distribution of state $S(i+1,t)$ knowing $S(i,t+1)=y_i$ is $U(y_i;.)$.
\end{enumerate}\par
\end{definition}

In the following, we study under which conditions a PCA admits a HZMC as invariant distribution. For $\mu$-supported PCA, the HZMC itself will be $\mu$-supported: a $(\rho_0,D,U)$-HZMC is $\mu$-supported if, $\rho_0 \ll \mu$ and $D$ and $U$ are $\mu$-supported. In that case, we denote $r_0$, $d$ and $u$ their respective $\mu$-densities. Hence, a $\mu$-supported $(\rho_0,D,U)$-HZMC is invariant by a $\mu$-supported PCA with t.k. $T$, if, for any $k\geq 0$, for $\mu$-almost $b_0,b_1,\dots,b_{k+1},c_0,\dots,c_{k} \in E$,
\begin{eqnarray} 
&& r_0(b_0) \left(\prod_{i=0}^{k} d(b_i;c_i) u(c_i;b_{i+1})\right) \nonumber \\
&& \quad \quad = \int_{E^{k+3}} r_0(a_0) \left(\prod_{i=0}^{k+1} d(a_i;b_i) u(b_i;a_{i+1})\right) \left(\prod_{i=0}^{k} t(b_i,b_{i+1};c_i) \right) \d \mu^{k+3}(a_0,\dots,a_{k+2})  \label{eq:brut}
\end{eqnarray}

The \defi{support} $\tilde{E}_{(\rho_0,D,U)}$ of a $(\rho_0,D,U)$-HZMC on $\HZ_{\mathbb{N}}(t)$ is the union of the support of the marginals of the first line of the HZMC, i.e. $\displaystyle \tilde{E}_{(\rho_0,D,U)} = \bigcup_{i \in \mathbb{N}} \text{supp}(\rho_i)$ where $\rho_i$ is the distribution of $S(i,t)$. When the $(\rho_0,D,U)$-HZMC is $\mu$-supported, then, for $\mu$-almost $x \in \tilde{E}_{(\rho_0,D,U)}$, there exists $i \in \mathbb{N}$ such that $r_i(x)>0$ (that holds because $E$ is a Polish space). In the case of a $\mu$-positive $(\rho_0,D,U)$-HZMC, $\tilde{E}_{(\rho_0,D,U)} = \text{supp}(\mu)$. When the context is clear, $\tilde{E}_{(\rho_0,D,U)}$ will be denoted $\tilde{E}$.\par

\begin{remarque}
Take a  $\mu$-supported PCA $\PCA{A}$ with t.k. $T$ and a $\mu$-supported $(\rho_0,D,U)$-HZMC with support $\tilde{E}$. Suppose that the $(\rho_0,D,U)$-HZMC is invariant by $\PCA{A}$. Now take a $\mu$-supported PCA $\PCA{A'}$ with t.k. $T'$ such that, for any $a,b \in \tilde{E}$, $T'(a,b;.)=T(a,b;.)$. Then, the $(\rho_0,D,U)$-HZMC is also invariant by $\PCA{A'}$. Hence, to characterize if a HZMC with support $\tilde{E}$ is invariant by a PCA with t.k. $T$, the value of $T(a,b;.)$ for $a$ or $b$ not in $\tilde{E}$ are not necessary.  
\end{remarque}

Let $\mu$ be a measure on $E$ and  $d: (a,c) \mapsto d(a;c)$ and $u: (c,b) \mapsto u(c;b)$ be two $\mu^2$-measurable functions from $E^2$ to $\reals$, then the $\mu^2$-measurable function $\truc{du}$ from $E^2$ to $\reals$ is defined by $\displaystyle \truc{du}(a;b) = \int_E d(a;c) u(c;b) \d \mu(c)$. For a $\mu$-supported HZMC, $\truc{du}(a;b)$ is the $\mu$-density of the t.k. ($DU$) of the Markov chain (induced by the HZMC) on the first line $S_t=(S(i,t):~i \in \mathbb{N})$ of $\HZ_{\mathbb{N}}(t)$.
\par

\subsubsection*{Main results}
We start with a generalization to Polish space alphabets of Lemma~2.3 in~\cite{CasMar14}.
\begin{theoreme}\label{theo:Toom}
Let $\mu$ be a $\sigma$-finite measure on a general alphabet $E$. Let $\PCA{A} := (\mathbb{N}, E, N, T)$ be a $\mu$-supported PCA and $(\rho_0,D,U)$ a $\mu$-supported HZMC with support $\tilde{E}$. The $(\rho_0,D,U)$-HZMC is invariant by $\PCA{A}$ if and only if the three following conditions are satisfied:
\begin{cond}\label{cond:Toom}
for $\mu^3$-almost $(a,b,c) \in \tilde{E}^3$, $t(a,b;c) \truc{du}(a;b) = d(a;c) u(c;b)$,
\end{cond}
\begin{cond}\label{cond:commut}
for $\mu^2$-almost $(a,b) \in \tilde{E}^2$, $\truc{du}(a;b) = \truc{ud}(a;b)$,
\end{cond}
\begin{cond}\label{cond:invD}
the Markov chain with t.k. $D$ possesses $\rho_0$ as invariant distribution, i.e. for $\mu$-almost $c$, $\displaystyle r_0(c) = \int_E r_0(a) d(a;c) \d \mu(a)$.
\end{cond}
\end{theoreme}
\par

We arrive to our main Theorem~\ref{theo:main}. When a PCA with t.k. $T$ is $\mu$-positive, we can go further and reduce the existence of an invariant HZMC for the PCA to the existence of a function $\eta$ solution to a cubic integral equation on $T$. In case of existence, we can express the kernels of the invariant HZMC using $\eta$ and $T$. Let us first introduce some material.\par
Let $\PCA{A}$ be a PCA with t.k. $T$ whose $\mu$-density is $t$. Define, for any positive measurable function $\phi \in L^1(\mu)$ (i.e. for $\mu$-almost $x \in E$, $\phi(x) > 0$ and $\displaystyle \int_E \phi(x) \d \mu(x) < \infty$), the two $\mu^2$-measurable functions $d^\phi: E^2 \mapsto \reals$ and $u^\phi: E^2 \mapsto \reals$ by
\begin{equation}\label{eq:DUeta}
d^\phi(a;c) = \frac{\displaystyle \int_E \frac{\phi(x)}{t(a,x;c_0)} t(a,x;c) \d \mu(x)}{\displaystyle \int_E \frac{\phi(x)}{t(a,x;c_0)} \d \mu(x)} \text{ and } u^\phi(c;b) = \frac{\displaystyle \frac{\phi(b)}{t(a_0,b;c_0)} t(a_0,b;c)}{\displaystyle \int_E \frac{\phi(x)}{t(a_0,x;c_0)} t(a_0,x;c) \d \mu(x)}.
\end{equation}

\begin{theoreme} \label{theo:main}
Let $\mu$ be a $\sigma$-finite measure on a general alphabet $E$. Let $\PCA{A} := (\mathbb{N}, E, N, T)$ be a $\mu$-positive PCA. $\PCA{A}$ admits a $\mu$-positive invariant HZMC if and only if the three following conditions are satisfied:

\begin{cond}\label{cond:Belyaev}%
there exists a triplet $(a_0,b_0,c_0) \in E^3$ such that $T(a_0,b_0;.)$ and $\mu$ are positive equivalent and, for $\mu^3$-almost $(a,b,c)$,
\begin{displaymath} \label{eq:Belyaev}
t(a,b;c) t(a_0,b_0;c) t(a_0,b;c_0) t(a,b_0;c_0) = t(a_0,b_0;c_0) t(a,b;c_0) t(a,b_0;c) t(a_0,b;c),
\end{displaymath}
\end{cond}

\begin{cond}\label{cond:pfff}
there exists a positive function $\eta \in L^1(\mu)$ solution to: for $\mu^2$-almost $(a,b)$ and for the $(a_0,c_0)$ of \C~\ref{cond:Belyaev},
\begin{equation}\label{eq:pfff}
\frac{ \displaystyle \frac{\eta(b)}{t(a,b;c_0)} }{ \displaystyle \int_E \frac{\eta(x)}{t(a,x;c_0)} \d \mu(x) } = \int_E \frac{ \displaystyle \frac{\eta(c)}{t(a_0,c;c_0)} t(a_0,c;a)}{ \displaystyle \int_E \frac{\eta(x) }{ t(c,x;c_0)} \d \mu(x) } \frac{ \displaystyle \int_E \frac{\eta(x)}{t(c,x;c_0)} t(c,x;b) \d \mu(x)}{ \displaystyle \int_E \frac{\eta(x)}{t(a_0,x;c_0)} t(a_0,x;a) \d \mu(x)} \d \mu(c),
\end{equation}
\end{cond}

\begin{cond}\label{cond:invDex}
the Markov chain with t.k. $D^\eta$, whose $\mu$-density is $d^\eta$ given by Eq~(\ref{eq:DUeta}), possesses as invariant distribution a probability distribution $\rho_0$.
\end{cond}

In this case, the $(\rho_0,D^\eta,U^\eta)$-HZMC where $D^\eta$ and $U^\eta$ are t.k. of $\mu$-densities given by Eq~(\ref{eq:DUeta}) is invariant by $\PCA{A}$. 
\end{theoreme}

\begin{remarque}
If \C~\ref{cond:Belyaev} and \C~\ref{cond:pfff} hold and if $E$ is finite, the Markov chain with t.k. $D^\eta$ is irreducible and aperiodic (because, for any $a,c \in E$, $D^\eta(a,c)>0$) and, so, it possesses a unique invariant distribution, i.e. \C~\ref{cond:invDex} always holds. If $E$ is not finite, we refer the reader to the book of Meyn and Tweedie~\cite{Meyn09} to get some conditions on $D^\eta$ for which the Markov chain with t.k. $D^\eta$ possesses an invariant distribution.
\end{remarque}

When the alphabet is finite, we can go further and show that if $\eta$ satisfies Eq~(\ref{eq:pfff}) then $\eta$ is the eigenvector of a computable matrix, obtaining such a way a simple and strong condition for the existence of such $\eta$ (this is done in~\cite{CasMar14}). For PCA with a general alphabet, this can not be done due to measurability issues that, roughly, do not allow us to take $a=b$ in Eq~(\ref{eq:pfff}). Nevertheless, under stronger conditions on $t$, we can characterize a set of functions that contains the set of functions $\eta$ solution to Eq~(\ref{eq:pfff}).
\begin{proposition}\label{prop:diag}
Let $\mu$ be a $\sigma$-finite measure on a general alphabet $E$. Let $\PCA{A} := (\mathbb{Z}, E, N, T)$ be a $\mu$-positive PCA. Suppose that \C~\ref{cond:Belyaev} and the two following conditions are satisfied:

\begin{cond}\label{cond:BelyaevDiag}%
for the same triplet $(a_0,b_0,c_0)$ of \C~\ref{cond:Belyaev}, for $\mu^2$-almost $(a,c)$,  
\begin{equation} \label{eq:BelyaevDiag}
t(a,a;c) t(a_0,b_0;c) t(a_0,a;c_0) t(a,b_0;c_0) = t(a_0,b_0;c_0) t(a,a;c_0) t(a,b_0;c) t(a_0,a;c),
\end{equation}
\end{cond}

\begin{cond}\label{cond:pfffDiag}
there exists a positive function $\eta \in L^1(\mu)$ solution to: for $\mu$-almost $a$ and for the $(a_0,c_0)$ of \C~\ref{cond:Belyaev},
\begin{equation}
\frac{ \displaystyle \frac{\eta(a)}{t(a,a;c_0)} }{ \displaystyle \int_E \frac{\eta(k)}{t(a,x;c_0)} \d \mu(x) } = \int_E \frac{ \displaystyle \frac{\eta(c)}{t(a_0,c;c_0)} t(a_0,c;a)}{ \displaystyle \int_E \frac{\eta(x) }{ t(c,x;c_0)} \d \mu(x) } \frac{ \displaystyle \int_E \frac{\eta(x)}{t(c,x;c_0)} t(c,x;a) \d \mu(x)}{ \displaystyle \int_E \frac{\eta(x)}{t(a_0,x;c_0)} t(a_0,x;a) \d \mu(x)} \d \mu(c).
\end{equation}
\end{cond}

Then, $\eta$ is a positive eigenfunction of 
\begin{displaymath}
\mathcal{A}_2 : f \mapsto \left( \mathcal{A}_2(f) : a \mapsto \int_E f(k) \frac{t(a,a;c_0)}{t(a,x;c_0)} \nu(a) \d \mu(x) \right)
\end{displaymath}
where $\nu$ is a positive eigenfunction (unique up to a multiplicative constant) in $L^1(\mu)$ of
\begin{displaymath}
\mathcal{A}_1 : f \mapsto \left( \mathcal{A}_1(f) : a \mapsto \int f(c) t(c,c;a) \d \mu(c) \right).
\end{displaymath}
\end{proposition}

\begin{remarque} \label{rem:positivePCA}
$\bullet$ Any positive PCA with finite alphabet $E$ (i.e. for all $a,b,c$, $T(a,b;c)>0$) is a $\mu_E$-positive PCA where $\mu_E$ is the counting measure on $E$. Hence, \C~\ref{cond:BelyaevDiag} and \C~\ref{cond:pfffDiag} are necessary implied by \C~\ref{cond:Belyaev} and \C~\ref{cond:pfff} in the case of finite alphabets. Moreover, in that case, $\mathcal{A}_1$ and $\mathcal{A}_2$ have their own unique eigenfunction (due to Perron-Frobenius theorem) and \C~\ref{cond:invDex} holds necessarily. So, applying Theorem~\ref{theo:main} and~Prop~\ref{prop:diag} to positive PCA give Theorem~2.6 in~\cite{CasMar14}.\par
$\bullet$ Let $E = \reals$ and $\mu$ be the Lebesgue measure. In the case where $t$ is \emph{continuous at any point} of $E^3$, then \C~\ref{cond:Belyaev} and \C~\ref{cond:pfff} imply \C~\ref{cond:BelyaevDiag} and \C~\ref{cond:pfffDiag} by continuity. And so a solution $\eta$ to Eq~(\ref{eq:pfff}) is a function $\eta$ given by Prop~\ref{prop:diag}.\par
$\bullet$ If for a PCA $\PCA{A}$ the conditions of Prop~\ref{prop:diag} do not hold, it is in general complex to find a function $\eta$ solution to Eq~(\ref{eq:pfff}). But, it may happen that a $\mu$-equivalent PCA $\PCA{A'}$ to $\PCA{A}$ (see Definition~\ref{def:muEquivalence} in Section~\ref{sec:prel}) satisfies the conditions of Prop~\ref{prop:diag}. Hence, in the best-case scenario, we can characterize, thanks to Prop~\ref{prop:diag}, a $(\rho_0,D^\eta,U^\eta)$-HZMC invariant by $\PCA{A'}$. And, by Prop~\ref{prop:sameInv} (in Section~\ref{sec:prel}), this HZMC is also invariant by $\PCA{A}$. An application of this method is shown in Section~\ref{sec:GaussianPCA} where it is proved that AR(1)~process is an invariant distribution of $\PCA{\tilde{G}}_{m,\sigma}$ (defined in Example~\ref{ex:GaussianPCAd}). 
\end{remarque}

The uniqueness (up to a multiplicative constant) of the eigenfunction $\nu$ (in Prop~\ref{prop:diag}) is a consequence of the following lemma.
\begin{lemme}[Theorem~6.8.7 of Durrett~\cite{Dur10}] \label{lem:Durrett}
Let $\mathcal{A}:~f \mapsto \left( \mathcal{A}(f) : y \mapsto \displaystyle \int_E f(x) m(x;y) \mu(\d x) \right)$ be an integral operator of kernel $m$. If $m$ is the $\mu$-density of a $\mu$-positive t.k. $M$ from $E$ to $E$, then $\mathcal{A}$ possesses at most one positive eigenfunction in $L^1(\mu)$ (up to a multiplicative constant).
\end{lemme}

\subsubsection*{Content}
In Section~\ref{sec:prel}, we recall some facts about \RN{} theorem and, then, state some properties of $\mu$-supported and $\mu$-positive PCA.\par

Section~\ref{sec:ex} is dedicated to some examples of PCA. In Section~\ref{sec:PCAfa}, we show applications of Theorems~\ref{theo:Toom} and~\ref{theo:main} and Prop~\ref{prop:diag} to PCA with finite alphabets. In Section~\ref{sec:GaussianPCA}, we use Theorem~\ref{theo:main} and Prop~\ref{prop:diag} to show that the law of an autoregressive process of order $1$ (AR(1)~process) is invariant by both Gaussian PCA $\PCA{G}_{m,\sigma}$ and $\PCA{\tilde{G}}_{m,\sigma}$ (defined in Example~\ref{ex:GaussianPCA} and~\ref{ex:GaussianPCAd}). In Section~\ref{sec:BetaPCA}, we present a Lebesgue-supported PCA called Beta PCA. In Section~\ref{sec:PCAlit}, we present first a PCA with alphabet $\reals$ that simulates a synchronous TASEP on $\mathbb{R}$ as defined by Blank~\cite{Blank12} and, then, a PCA with alphabet $\reals$ that simulates the first-passage percolation as presented by Kesten~\cite{Kesten87} on a particular graph $\mathcal{G}$. Unfortunately, Theorem~\ref{theo:Toom} and~\ref{theo:main} do not apply to these two PCA.\par

In Section~\ref{sec:proof}, Theorems~\ref{theo:Toom} and~\ref{theo:main} and Prop~\ref{prop:diag}, the main contributions of the paper, are proved.\par

Section~\ref{sec:extension} is devoted to extensions of Theorems~\ref{theo:Toom} and~\ref{theo:main} for PCA on $\mathbb{Z}$ and $\cyl{n}$. First, we extend in both cases the notion of HZMC: $\text{HZMC}_{\mathbb{Z}}$ on $\mathbb{Z}$ and cyclic-HZMC (CHZMC) on $\cyl{n}$ (if $E$ is finite, a CHZMC is an HZMC conditioned to be periodic and, in the general case, it is a Gibbs measure). Then, we characterize PCA letting $\text{HZMC}_{\mathbb{Z}}$ invariant, and also PCA letting CHZMC invariant.\par

\section{Preliminaries}\label{sec:prel}
We recall here some facts around \RN{} theorem.\par

Let $\mu$ and $\nu$ be two measures on $E$. $\mu$ is \defi{equal} to $\nu$ ($\mu=\nu$) if for all $A \in \borel{E}$, $\mu(A)=\nu(A)$. $\mu$ is \defi{absolutely continuous} with respect to $\nu$ ($\mu \ll \nu$) if, for all $A \in \borel{E}$, $\nu(A)=0 \Rightarrow \mu(A)=0$. And $\mu$ and $\nu$ are \defi{singular} ($\mu \perp \nu$) if there exists $N \in \borel{E}$ such that $\mu(N)=0$ and $\nu(\comp{N})=0$. The \defi{\RN{} theorem} allows one to decompose a $\sigma$-finite measure with respect to an other. Let $\nu$, $\mu$ be two positive $\sigma$-finite measures on $E$, then there exists a unique pair of positive $\sigma$-finite measures $(\mu_1,\mu_2)$ such that $\mu=\mu_1+\mu_2$ with $\mu_1 \ll \nu$ and $\mu_2 \perp \nu$. Moreover, there exists a unique (up to a $\nu$-null set) $\nu$-measurable function $f : E \longrightarrow \reals^{+}$ such that, for all $A \in \borel{E}$, $\displaystyle \mu_1(A) = \int_A f \d \nu$. The function $f$ is denoted $\displaystyle \frac{\d \mu}{\d \nu}$ and called \RN{}-derivative of $\mu$ with respect to $\nu$ (or $\nu$-density).\par

\begin{definition}[Positive equivalence]
Let $\nu$, $\mu$ be two measures on $E$. $\nu$ and $\mu$ are \defi{positive equivalent} if $\nu \ll \mu$ and $\mu \ll \nu$. In that case, $\displaystyle \frac{\d \mu}{\d \nu} > 0$ and $\displaystyle \frac{\d \nu}{\d \mu}>0$, $\mu$-almost everywhere.
\end{definition}\par
 
Now, we give some properties of $\mu$-positive PCA and define the $\mu$-equivalence of PCA.\par

\begin{proposition} \label{prop:MuNuPos}
Let $\PCA{A}$ be a PCA. If $\PCA{A}$ is $\mu$ and $\nu$-positive, then $\mu$ and $\nu$ are positive equivalent or singular. 
\end{proposition}

\begin{proof}
Let $\PCA{A}$ be a PCA that is both $\mu$ and $\nu$-positive. If there exists $(a,b) \in E^2$ such that the measure $T(a,b;.)$ is both $\mu$ and $\nu$-positive then $\nu$ is $\mu$-positive by transitivity. Else, $P_{\mu} = \{ (a,b):~T(a,b;.) \text{ is $\mu$-positive}\}$ and $P_{\nu} = \{ (a,b):~T(a,b;.) \text{ is $\nu$-positive}\}$ are measurable and disjoint, and so taking $N=P_{\nu} \subset \comp{P_{\mu}}$, $\mu(N)=0$ and $\nu(\comp{N})=0$.
\end{proof}

The PCA $(\PCA{\tilde{G}}_{m,\sigma})$ (defined in Example~\ref{ex:GaussianPCAd}) are Lebesgue-positive; indeed, $\{(a,b):~T(a,b;.) \not\mskip-\thinmuskip\ll \text{Lebesgue-measure}\} = \{(a,a):~a \in \reals\}$ is Lebesgue-negligible in $\reals^2$. Moreover, for any $a \in \reals$, they are $\delta_a$-measurable because $T(a,a;.) = \delta_a$. One can verify that Prop~\ref{prop:MuNuPos} holds for these PCA because $\{\delta_a:~a \in \reals\}$ and the Lebesgue-measure are pairwise singular.
 
\begin{definition}[$\mu$-equivalent PCA] \label{def:muEquivalence}
Let $\PCA{A}$ and $\PCA{A'}$ be two $\mu$-supported PCA with respective t.k. $T$ and $T'$. $\PCA{A}$ and $\PCA{A'}$ are said to be \defi{$\mu$-equivalent} if the set where $T$ and $T'$ are not equal is a $\mu^2$-negligible set, i.e. $\mu^2\left(\{(a,b):~T(a,b;.) \neq T'(a,b;.)\}\right) = 0$.
\end{definition}

For any $(m,\sigma)$, the Gaussian PCA $\PCA{G}_{m,\sigma}$ (defined in Example~\ref{ex:GaussianPCA}) and the PCA $\PCA{\tilde{G}}_{m,\sigma}$ are Lebesgue-equivalent (their t.k. differs on $\{(a,a):~a \in \reals\}$, a Lebesgue-negligible set).

\begin{proposition}\label{prop:sameInv}
Let $\PCA{A}$ and $\PCA{A'}$ be two $\mu$-equivalent PCA and $(\rho_0,D,U)$ a $\mu$-supported HZMC. If $(\rho_0,D,U)$ is an invariant measure for $\PCA{A}$, then $(\rho_0,D,U)$ is also an invariant measure for $\PCA{A'}$.
\end{proposition}
\begin{proof}
By property of $\mu$-equivalent PCA, we can change $t$ by $t'$ in Eq~(\ref{eq:brut}).
\end{proof}

Hence, sometimes, to find an invariant HZMC of a $\mu$-supported PCA $\PCA{A}$, the easiest way is to find a $\mu$-equivalent PCA $\PCA{A'}$ for which we already know a $\mu$-positive invariant HZMC. In particular, for a $\mu$-positive PCA $\PCA{A}$ for which Prop~\ref{prop:diag} does not apply, it could exist a $\mu$-equivalent PCA $\PCA{A'}$ for which this Proposition applies and gives a solution $\eta$ to Eq~(\ref{eq:pfff}). This Proposition gives some ``degrees of freedom'' on the ``rigid'' integral cubic equation Eq~(\ref{eq:pfff}). In Section~\ref{sec:GaussianPCA}, this Proposition will be used to prove that an invariant measure to $\PCA{\tilde{G}}_{m,\sigma}$ is an AR(1)~process.\par

\section{Examples} \label{sec:ex}
\begin{notation}
In this section, if $E$ is a finite set, then $\displaystyle \mu_E = \sum_{x \in E} \delta_x$ is the counting measure on $E$.\par
\end{notation}

Our first examples are PCA with finite alphabets. Then, we introduce two new models: Gaussian PCA and Beta PCA to illustrate our theorems. Finally, we present PCA with infinite alphabets that model existing problems in literature: one PCA models a synchronous TASEP on $\reals$ as defined by Blank~\cite{Blank12} and an other one a variant of directed first-passage percolation.\par

All PCA presented in this section are PCA on $\mathbb{N}$ (except the PCA modeling TASEP that is on $\mathbb{Z}$) and neighborhood $N(i)=(i,i+1)$.

\subsection{PCA with finite alphabet} \label{sec:PCAfa}
For positive PCA, see the first point of Remark~\ref{rem:positivePCA}.

In the following example, we focus on PCA that are not positive and take a PCA not $\mu_E$-positive, but $\mu_F$-positive for some $F$ subsets of $E$.\par

Let $\PCA{A}$ be the PCA with alphabet $E = \{0,1,2\}$ and t.k.:
\begin{itemize}
  \setlength{\itemsep}{1pt}
  \setlength{\parskip}{0pt}
  \setlength{\parsep}{0pt}
\item $T(0,i;i) = T(i,0;i) = 1$ for all $i \in \{0,1,2\}$,
\item $T(1,1;1) = T(1,1;2) = T(2,2;1) = T(2,2;2) = 1/2$,
\item $T(1,2;1) = T(2,1;2) = 4/5$,
\item $T(1,2;2) = T(2,1;1) = 1/5$.
\end{itemize}

This PCA is not positive ($T(0,1;0) = 0$), nevertheless it is $\mu_{\{0\}}$-positive ($T(0,0;.)=\mu_{\{0\}}(.)$) and, also, $\mu_{\{1,2\}}$-positive. These two measures are singular as ``predicted'' by Prop~\ref{prop:MuNuPos}.\par

Considered as a $\mu_{\{0\}}$-positive PCA, Theorem~\ref{theo:mainZ} and Prop~\ref{prop:diag} to $\PCA{A}$ imply that the constant (equals to $0$) HZMC is invariant by $\PCA{A}$.\par

Application of the same Theorem and same Lemma when $\PCA{A}$ is considered as a $\mu_{\{1,2\}}$-positive PCA gives: that \C~\ref{cond:Belyaev} holds because 
\begin{displaymath}
T(1,1;1) T(2,2;1) T(1,2;2) T(2,1;2) = T(2,2;2) T(1,1;2) T(2,1;1) T(1,2;1) = 1/25,
\end{displaymath}
for $(a_0,b_0,c_0)$ we take $(1,1,1)$; 
then we obtain for $\nu$ and $\eta$, as defined in Prop~\ref{prop:diag}, $\displaystyle \nu(1) = \nu(2) = 1/2$, $\displaystyle \eta(1) = 1/3$ and  $\displaystyle \eta(2) = 2/3$ and, so, taking $D^\eta$ and $U^\eta$ as defined in Eq~(\ref{eq:DUeta}): $\displaystyle D^\eta(1;1) = D^\eta(2;2) = 2/3$, $D^\eta(1;2) = D^\eta(2;1) = 1/3$, $U^\eta(1;1) = U^\eta(2;2) = 1/3$ and $U^\eta(1;2) = U^\eta(2;1) = 2/3$. Then, $\displaystyle \rho_0(1) = \rho_0(2) = 1/2$ is an invariant measure for the Markov chain of kernel $D^\eta$. Hence, the are two HZMC, supported by two singular measures ($\mu_{\{0\}}$ and $\mu_{\{1,2\}}$), invariant by $\PCA{A}$.

\subsubsection*{A $\mu$-supported PCA}
Let $\PCA{A}$ be the PCA with alphabet $E = \cyl{\kappa}$ with t.k. $T$ such that $T(a,b;.)$ is the uniform distribution on the circular interval set $\{a+1, \dots, b-1\}$ if $\module{a-b}>1$ and if $\module{a-b}\leq 1$, it is the uniform distribution on $E$.\par

This PCA is a $\mu_E$-supported PCA, but not $\mu$-positive for any measure $\mu$ on $E$. This PCA $\PCA{A}$ has an invariant $(\rho_0,D,U)$ HZMC with $D(a;a+1 \mod \kappa) = U(a;a+1 \mod \kappa) = 1$ for all $a \in \cyl{\kappa}$ and for any $a \in \cyl{\kappa}$, $\displaystyle \rho_0(\kappa) = \frac{1}{\kappa}$.

\subsection{Two new models of PCA with infinite alphabet}
\subsubsection{Gaussian PCA} \label{sec:GaussianPCA}
\begin{notation}
In the following, for any two positive parameters $m$ and $\sigma$, the Lebesgue-density of the Gaussian distribution of mean $m$ and variance $\sigma^2$ will be denoted
\begin{displaymath}
g[m,\sigma](x) = \frac{1}{\sqrt{2 \pi \sigma^2}} \exp \left( - \frac{(x-m)^2}{2 \sigma^2} \right).
\end{displaymath}
\end{notation}

In this section, we apply Theorem~\ref{theo:main} and Prop~\ref{prop:diag} to prove that an AR(1)~process is an invariant distribution for Gaussian PCA $\PCA{G}_{m,\sigma}$ (defined in Example~\ref{ex:GaussianPCA}). Then, we prove the same property for PCA $\PCA{\tilde{G}}_{m,\sigma}$ (defined in Example~\ref{ex:GaussianPCAd}) by an application of Prop~\ref{prop:sameInv}.

\paragraph{Gaussian PCA $\PCA{G}_{m,\sigma}$.} For $\PCA{G}_{m,\sigma}$, it can be checked that \C~\ref{cond:Belyaev} holds for any triplet $(a_0,b_0,c_0)$ in $\reals^3$, so let us choose $(a_0,b_0,c_0) = (0,0,0)$. We use Prop~\ref{prop:diag} to obtain a function $\eta$. The first step consists in studying the eigenfunctions of
\begin{displaymath}
\app{\mathcal{A}_1}{L^1}{L^1}{f}{\mathcal{A}_1(f):c \mapsto \displaystyle \int_\reals f(a) \  g\left[\frac{2}{m}a,\sigma\right](c) \ \d a}.
\end{displaymath}
The function $\displaystyle \nu(x) = \exp \left( - \frac{1}{2\sigma^2} \left(1 - \frac{4}{m^2}\right) x^2 \right)$ is a positive eigenfunction of $\mathcal{A}_1$. Moreover, we need $\nu$ to be in $L^1$, hence $\displaystyle 1 - \frac{4}{m^2}$ must be positive and, so, we need $\module{m}>2$. Without this condition, for any $i$, the function $t \to \text{Var}\left(S(i,t)\right)$ increases and goes to infinity with $t$. 
When $\module{m}>2$, we can go further with Prop~\ref{prop:diag} and study the eigenfunctions of
\begin{displaymath}
\app{\mathcal{A}_2}{L^1}{L^1}{f}{\mathcal{A}_2(f) : a \mapsto \displaystyle \int_\reals f(b) \frac{t(a,a;0)\nu(a)}{t(a,b;0)} \d b}
\end{displaymath}
with $\displaystyle \frac{t(a,a;0)\nu(a)}{t(a,b;0)} = \exp \left( -\frac{b^2}{2 \sigma^2} \right) \exp \left( \frac{\left(\frac{a+b}{m}\right)^2}{2\sigma^2} \right)$. One can check that the function 
\begin{displaymath}
\eta(x) = \exp \left(-  \frac{1}{4\sigma^2} \left( 1 + \sqrt{1 - \frac{4}{m^2}} \right) x^2 \right)
\end{displaymath}
is a positive eigenfunction of $\mathcal{A}_2$  associated to the eigenvalue $\displaystyle \frac{\sqrt{\pi \sigma^2}}{\left( 1 + \sqrt{1-\frac{4}{m^2}} \right)^2}$. Moreover $\eta$ satisfies Eq~(\ref{eq:pfff}) (this is an example where Prop~\ref{prop:diag} permits to compute a solution $\eta$ to Eq~(\ref{eq:pfff})). We get
\begin{equation}\label{eq:dGauss}
d^\eta(a;c) = g\left[\frac{2}{ml} \, a , \sqrt{\frac{2}{l}} \, \sigma \right](c)
\end{equation}
and
\begin{equation}\label{eq:uGauss}
u^\eta(c;b) =  g\left[\frac{2}{ml} \, c , \sqrt{\frac{2}{l}} \, \sigma \right](b)
\end{equation}
for $\displaystyle l = 1 + \sqrt{1 - \frac{4}{m^2}}$. To end, we need to find an invariant probability distribution $\rho_0$ for the Markov chain of t.k. $D^\eta$ (of Lebesgue-density $d^\eta$). The measure $\rho_0$ with the following Lebesgue-density $r_0$ is fine:
\begin{equation} \label{eq:rGauss}
r_0(x) = g\left[0, \left( 1 - \frac{4}{m^2} \right)^{-1/4} \sigma \right](x)
\end{equation}
\par

This permits to conclude that the $(\rho_0,D^\eta,U^\eta)$-HZMC is an invariant measure for the Gaussian PCA.\par

In fact, this invariant HZMC is an autoregressive process of order $1$ (AR(1)~process, see~\cite{West97}) that is a process $(X_i)$ such that $X_i = \theta + \phi X_{i-1} + \epsilon_i$ where $\theta$ and $\phi$ are two real numbers and $\left( \epsilon_i \right)$ are independent and identically distributed of law the Gaussian law $\mathcal{N}(0,\sigma'^2)$. In our case, the invariant HZMC is an AR(1)~process on $\HZ_{\mathbb{N}}$ with $\theta = 0$, $\phi = \displaystyle \frac{2}{ml}$ and $\displaystyle \sigma'^2 = \frac{2 \sigma^2}{l}$.

\paragraph*{``Gaussian PCA except on diagonal'' $\PCA{\tilde{G}}_{m,\sigma}$.}
As already seen in Section~\ref{sec:prel}, this PCA is Lebesgue-positive and also $\mu_{\{a\}}$-positive for any $a \in \reals$.\par

When we consider $\PCA{\tilde{G}}_{m,\sigma}$ as a Lebesgue-positive PCA, Prop~\ref{prop:diag} could not be used to find a solution $\eta$ to Eq~(\ref{eq:pfff}). Hopefully, $\PCA{\tilde{G}}_{m,\sigma}$ is Lebesgue-equivalent to $\PCA{G}_{m,\sigma}$. Hence, by Prop~\ref{prop:sameInv}, the invariant Lebesgue-positive $(\rho_0,D^\eta,U^\eta)$-HZMC, that corresponds to an AR(1)~process, obtained for $\PCA{G}_{m,\sigma}$ is also invariant for $\PCA{\tilde{G}}_{m,\sigma}$.\par
Besides, for any $a \in \reals$, the constant process equal to $a$ everywhere is also an invariant measure to $\PCA{\tilde{G}}_{m,\sigma}$.

\subsubsection{Beta PCA}\label{sec:BetaPCA}
We define a class of PCA with alphabet $\reals$ depending on three positive real parameters $\alpha$, $\beta$ and $m$. The t.k. is the following: for all $a,b \in \reals$ and $C \in \borel{\reals}$,
\begin{displaymath}
T(a,b;C) = \prob{(b-a) X + a - m \in C} 
\end{displaymath} 
where $X$ is a $\text{Beta}(\alpha,\beta)$ random variable, i.e. the Lebesgue-density of $T$ is, for $\mu$-almost $a,b,c$,
\begin{displaymath}
t(a,b;c) = \frac{\displaystyle \left( \frac{c+m-a}{b-a} \right)^{\alpha-1} \left(\frac{b-c-m}{b-a} \right)^{\beta-1} }{B(\alpha,\beta)} \indic{0 \leq \frac{c + m - a}{b-a} \leq 1}
\end{displaymath}
where $B$ is the beta function. In words, the PCA takes a random (following a Beta law) number between the two values of its two neighbors and subtract $m$ to it.\par

This PCA is Lebesgue-supported, but not Lebesgue-positive.\par

Now, try to search an invariant $(\rho_0,D,U)$-HZMC to this PCA. Let $\theta$ be a positive real number. Let $D_1(a;C) = \prob{X_1 + a - m \in C}$ and $U_1(c;B) = \prob{X_2 + c + m \in B}$ where $X_1$ (resp. $X_2$) is a $\text{Gamma}(\alpha,\theta)$ (resp. $\text{Gamma}(\beta,\theta)$) random variable.
For $D = D_1$ and $U = U_1$, \C~\ref{cond:Toom} and \C~\ref{cond:commut} hold; unfortunately, there does not exist a probability distribution $\rho_0$ that satisfies \C~\ref{cond:invD}. Hence, this PCA does not possess a Lebesgue-supported HZMC as invariant distribution. Nevertheless, the image of a Lebesgue-supported $(\rho,D_1,U_1)$-HZMC by this PCA is the $(\rho D_1,D_1,U_1)$-HZMC (meaning that one can describe simply the distribution of the successive image of a $(\rho,D_1,U_1)$-HZMC by $A$).

\subsection{PCA with infinite alphabet in the literature} \label{sec:PCAlit}
\subsubsection*{PCA modeling TASEP}
We model the synchronous TASEP on $\reals$ introduced by Blank~\cite{Blank12} by a PCA on $\mathbb{Z}$ with alphabet $\reals$. In the following, when we say TASEP, we refer to this variant of TASEP.

TASEP models the behavior of an infinite number of particles of radius $r \geq 0$ on the real line, that move to the right direction, that do not bypass, not overlap and, at each step of time, each particle moves with probability $p$ ($0 < p \leq 1$), independently of each others. When a particle moves, it travels a distance $v \geq 0$ to the right direction, except if it can create a collision with the next particle, in that case, it moves to the rightest allowed position. Formally, the evolution of $(x_i^t)$ is the following:
\begin{displaymath}
x_i^{t+1} = \begin{cases} \min(x_i^t +v , x_{i+1}^t -2r) & \text{with probability } p, \\
 x_i^t & \text{with probability } 1-p.
\end{cases}
\end{displaymath}

We propose, here, to model this TASEP by a PCA $\PCA{A}$ on $\mathbb{Z}$ with alphabet $\reals$. In this model, the state of a cell $i$ at time $t$ is the position $x_i^t$ of the $i$th particle of the TASEP at time $t$. Hence, the t.k. of the PCA is the following: for any $a,b \in \reals$ such that $a + r \leq b - r$ and for any $C \in \borel{\reals}$,
\begin{displaymath} 
T(a,b;C) = \begin{cases} 
(1-p) \delta_a(C) + p \delta_{a+v}(C) & \text{if } a+v \leq b-2r,\\
(1-p) \delta_a(C) + p \delta_{b-2r}(C) & \text{if } a+v > b-2r.
\end{cases}
\end{displaymath}
The t.k. for other pairs $(a,b)$ is not specified since they concern forbidden configurations. Hence, if we start with an admissible configuration at time $0$ for the PCA (i.e. for any $i \in \mathbb{Z}$, $S(i,0) + r \leq S(i+1,0) - r$), then the PCA models the TASEP.\par

We can remark that if, at some time $t$, $v=2kr$ for some $k \in \mathbb{N}$, and, for any $i$, $x_i(t) \in 2r \mathbb{Z}$, then at time $t+1$ this is also the case. In terms of PCA, this says that the PCA $\PCA{A}$ is $\mu$-supported by $\mu = \sum_{i \in \mathbb{Z}} \delta_{2ri}$. For this measure, one can check that the $(R,D,U)$-$\text{HZMC}_{\mathbb{Z}}$ (HZMC on $\mathbb{Z}$ are defined in Section~\ref{sec:PCAonZ}) where $R_i = \delta_{2ri}$, $D(a;a) = 1$ and $U(a;a+2r)=1$ (i.e. the states of the $\text{HZMC}_{\mathbb{Z}}$ are $S(i,t)=S(i,t+1)=2ri$ for all $i \in \mathbb{Z}$) is an invariant $\text{HZMC}_{\mathbb{Z}}$ for the PCA. But, it is quite an uninteresting invariant measure because it corresponds to a trivial configuration where nobody can move.\par

\subsubsection*{PCA modeling a variant of first-passage percolation}
We propose a model of a directed first-passage percolation $P$ on a directed graph $\mathcal{G}$ which can also be seen as a PCA with alphabet $[0,\infty)$. We use the same notation as Kesten~\cite{Kesten87} to present the classical model of first-passage percolation.\par

The set of nodes of $\mathcal{G}$ is $\mathcal{N} = \{(i,j):~i,j \in \mathbb{N}\}$, the discrete quarter plan, and the set of directed edges is $\mathcal{E} = \{((i,j),(i,j+1)):~i,j \in \mathbb{N} \} \cup \{((i+1,j),(i,j+1)):~i,j \in \mathbb{N}\}$. We denote $L_0$ the set of the nodes of the first line $L_0 = \{(i,0):~i \in \mathbb{N}\}$. Now, assign to each edge $e \in \mathcal{E}$ a random non-negative weight $t(e)$ which could be interpreted as the time needed to pass through the edge $e$. We assume that $\left( t(e):~e \in \mathcal{E} \right)$ are i.i.d. with common distribution $F$. The passage time of a directed path $r=(e_1,\dots,e_n)$ on $G$ is $T(r) = \displaystyle \sum_{i=1}^n t(e_i)$. The travel time from a node $u$ to a node $v$ is defined as $T(u,v) = \inf \{ T(r):~r \text{ is a directed path from $u$ to $v$} \}$. If there is no directed path from $u$ to $v$, $T(u,v) = \infty$. We define the travel time from a set of nodes $U$ to a node $v$ by $T(U,v) = \inf \{T(u,v):~u \in U\}$. Finally, we define $\mathcal{V}(t) = \{ v \in \mathcal{N}:~T(L_0,v) \leq t \}$ the set of visited nodes at time $t$. The object of study in the first-passage percolation is this set $\mathcal{V}(t)$.\par

The first-passage percolation $P$ on $\mathcal{G}$ can be seen as a PCA $\PCA{A}$ on $\mathbb{N}$ with alphabet $[0,\infty)$ as follows: let $S(i,j)$ represents the travel time $T(L_0,(i,j))$ from $L_0$ to the node $(i,j)$ in the first-passage percolation. Hence, the t.k. of the PCA is the following: for any $a,b \in [0,\infty)$, for any $C \in \mathcal{B}([0,\infty))$,
\begin{displaymath} 
T(a,b;C) = L_{a,b}(C)
\end{displaymath}
where $L_{a,b}$ is the distribution of the random variable $X = \min \{(a+T_1),(b+T_2)\}$ where $T_1$ and $T_2$ are i.i.d. with common law $F$.\par

Unfortunately, our work does not apply to these examples.

\section{Proofs of the main results} \label{sec:proof}
\subsection{Proof of Theorem~\ref{theo:Toom}} \label{sec:proofToom}
First: let $(\rho_0,D,U)$ be a $\mu$-supported HZMC invariant by $\PCA{A}$ with t.k. $T$, a $\mu$-supported PCA. For all $A,B,C \in \borel{E}$, for all $i \in \mathbb{N}$,
\begin{eqnarray*}
\prob{S(i,t) \in A, S(i+1,t) \in B, S(i,t+1) \in C} 
&=&\int_{A \times B \times C} r_i(a) d(a;c) u(c;b) \d \mu^3(a,b,c)\label{eq:1} \\ 
&=&\int_{A \times B \times C} r_i(a) \truc{du}(a;b) t(a,b;c) \d \mu^3(a,b,c) \label{eq:2}
\end{eqnarray*}
where $\rho_i$ is the law of cell $i$ of $\mu$-density $r_i$.
Taking the difference, we obtain, for all $A,B,C \in \borel{E}$,
\begin{displaymath}
\int_{A \times B \times C} \Big( r_i(a) d(a;c) u(c;b) -  r_i(a) \truc{du}(a;b) t(a,b;c) \Big) \d \mu^3(a,b,c)= 0.
\end{displaymath}
Hence, since this holds for any Borel set $A \times B \times C$, $r_i(a) d(a;c) u(c;b) = r_i(a) \truc{du}(a;b) t(a,b;c)$ for $\mu^3$-almost $(a,b,c) \in E^3$. If $a \in \tilde{E}$, there exists $i$ such that $r_i(a) > 0$ a.s. and, then, \C~\ref{cond:Toom} holds.

We have also, for all $A,B \in \borel{E}$, on one hand, 
\begin{eqnarray*}
\prob{S(i,t+1) \in A , S(i+1,t+1) \in B} & = & \prob{S(i,t+1) \in A , S(i+1,t+1) \in B, S(i+1,t) \in E} \\
& = & \displaystyle \int_{A \times B} r_i(a) \truc{ud}(a;b) \d \mu^2(a,b)
\end{eqnarray*}
because $\left( S(0,t),S(0,t+1),S(1,t), \dots \right)$ is a $(\rho_0,D,U)$-HZMC and, on the other hand,
\begin{eqnarray*}
\prob{S(i,t+1) \in A , S(i+1,t+1) \in B} & = & \prob{S(i,t+1) \in A , S(i,t+1) \in B, S(i,t+2) \in E} \\
& = & \displaystyle \int_{A \times B} r_i(a) \truc{du}(a;b) \d \mu^2(a,b)
\end{eqnarray*}
because $\left( S(0,t+1), S(0,t+2), S(1,t+1), \dots \right)$ is also a $(\rho_0,D,U)$-HZMC due to its invariance by $\PCA{A}$. Then, as before, $r_i(a) \truc{ud}(a;b) = r_i(a) \truc{du}(a;b)$ for $\mu^2$-almost $(a,b) \in E^2$ and, so, \C~\ref{cond:commut} holds.

Moreover, the law of $S(0,t)$ and $S(0,t+1)$ must be the same because $(\rho_0,D,U)$ is invariant by the PCA. Hence, the law of $S(0,t+1)$ of $\mu$-density $\displaystyle \int_E r_0(a) d(a;c) \d \mu(a)$ must be equal to $\rho_0$ of $\mu$-density $r_0(c)$, i.e. \C~\ref{cond:invD} holds.

Conversely, suppose that \C~\ref{cond:Toom}, \C~\ref{cond:commut} and \C~\ref{cond:invD} are satisfied. Suppose that the horizontal zigzag $\HZ_{\mathbb{N}}(t)$ is distributed as a $(\rho_0,D,U)$-HZMC. Now, compute the push forward measure of this HZMC by $\PCA{A}$. For any $n\geq 0$, for any $F_{2n+1} = B_0 \times \dots \times B_{n+1} \times C_0 \times \dots \times C_n \in \borel{E}^{2n+1}$.
\begin{eqnarray*}
\lefteqn{
\prob{S(0,t+1) \in B_0, S(0,t+2) \in C_0, \dots, S(n+1,t+1) \in B_{n+1}}} \\
& = & \int_{E^{n+2} \times F_{2n+1}} r_0(a_0) \prod_{i=0}^{n+1}d(a_i;b_i) u(b_i;a_{i+1})t(b_i,b_{i+1};c_i)\\
& & \qquad \qquad \qquad \qquad \qquad \qquad \d \mu^{3n+6}(a_0,\dots,a_{n+2},b_0,\dots,b_{n+1},c_0,\dots,c_n) \\
& = & \int_{F_{2n+1}} \left(\int_{E} r_0(a_0) d(a_0;b_0) \d \mu(a_0) \right) \prod_{i=0}^{n} \left( \int_{E} u(b_i;a_{i+1}) d(a_{i+1};b_{i+1}) \d \mu(a_{i+1}) \right)  \nonumber \\
& & \qquad \left(\int_{E} u(b_{n+1};a_{n+2}) \d \mu(a_{n+2}) \right) \prod_{i=0}^{n} t(b_i,b_{i+1};c_i) \d \mu^{2n+3}(b_0,\dots,b_{n+1},c_0,\dots,c_n) \\
& = & \int_{F_{2n+1}} r_0(b_0) \prod_{i=0}^{n} \truc{ud}(b_i;b_{i+1}) t(b_i,b_{i+1};c_i) \d \mu^{2n+3} (b_0,\dots,b_{n+1},c_0,\dots,c_n) \\
& = & \int_{F_{2n+1}} r_0(b_0) \prod_{i=0}^{n} d(b_i;c_i) u(c_i;b_{i+1}) \d \mu^{2n+3}(b_0,\dots,b_{n+1},c_0,\dots,c_n).
\end{eqnarray*}

This shows that the push forward measure of a $(\rho_0,D,U)$-HZMC is a $(\rho_0,D,U)$-HZMC. Hence, the $(\rho_0,D,U)$-HZMC is an invariant measure of $\PCA{A}$. 

\subsection{Proof of Theorem~\ref{theo:main}}
In the case of a $\mu$-positive HZMC, taking $\tilde{E}$ or $E$ does not make any difference in Theorem~\ref{theo:Toom}. Indeed, by basic properties of measurability: for any property $P$, $P(x)$ holds for $\mu$-almost $x \in E$ if and only if $P(x)$ holds for $\mu$-almost $x \in \text{supp}(\mu) \cap E$ (set equal to $\tilde{E}$ here). In addition, for a $\mu$-positive $(\rho_0,D,U)$-HZMC: for $\mu^2$-almost $(a,b) \in E^2$, $\truc{du}(a,b) > 0$.\par

To prove Theorem~\ref{theo:main}, we first prove Lemmas~\ref{lem:first} and~\ref{lem:second}.

\begin{lemme} \label{lem:first}
Let $\PCA{A}$ be a $\mu$-positive PCA with t.k. $T$. The three conditions: \C~\ref{cond:Toom}, \C~\ref{cond:Belyaev} and
 
\begin{cond}\label{cond:BelyaevGen}
For $\mu^6$-almost $(a,a',b,b',c,c')$,
\begin{equation} \label{eq:BelyaevGenProof}
t(a,b;c) t(a,b';c') t(a',b;c') t(a',b';c) = t(a',b';c') t(a',b;c) t(a,b';c) t(a,b;c').
\end{equation}
\end{cond}
are equivalent.
\end{lemme}

\begin{proof}
\begin{itemize}
\item From \C~\ref{cond:Toom} to \C~\ref{cond:BelyaevGen}: replace in \C~\ref{cond:BelyaevGen} the expressions of $t$ by the ones given in \C~\ref{cond:Toom}.

\item From \C~\ref{cond:BelyaevGen} to \C~\ref{cond:Belyaev}: we prove its contrapositive. Suppose that, for all $(a_0,b_0,c_0)$, \C~\ref{cond:Belyaev} is false. Hence, for all $(a_0,b_0,c_0) \in E^3$, either $T(a_0,b_0,.)$ and $\mu$ are not positive equivalent, or
\begin{equation} \label{eq:BelyaevProofPos}
\mu^3\left(\{(a,b,c) \text{ such that Eq~(\ref{eq:Belyaev}) does not hold}\}\right)>0.
\end{equation} 
But, by definition of $\mu$-positivity, the set of $(a_0,b_0)$ such that $T(a_0,b_0;.)$ and $\mu$ are not positive equivalent is $\mu^2$-negligible. Hence, for $\mu^3$-almost $(a_0,b_0,c_0)$, Eq~(\ref{eq:BelyaevProofPos}) holds. But, by Fubini theorem,
\begin{eqnarray*}
&&\mu^6(\{(a,b,c,a',b',c') \text{ such that Eq~(\ref{eq:BelyaevGenProof}) does not hold}\}) \\
&& = \int_{E^3} \mu^3\left(\{(a,b,c) \text{ such that Eq~(\ref{eq:Belyaev}) does not hold}\}\right) \d \mu(a_0,b_0,c_0)> 0
\end{eqnarray*}
and on the other hand \C~\ref{cond:BelyaevGen} is equivalent to
\begin{displaymath}
\mu^6(\{(a,b,c,a',b',c') \text{ such that Eq~(\ref{eq:BelyaevGenProof}) does not hold}\}) = 0.
\end{displaymath}

\item From \C~\ref{cond:Belyaev} to \C~\ref{cond:Toom}: set
\begin{displaymath}
  \displaystyle d(a;c) =  K_a \frac{t(a,b_0;c)}{t(a_0,b_0;c)} \frac{\displaystyle \int_E t(a_0,b;c) \d \mu(b)}{\displaystyle \int_E \frac{t(a,b_0;x)}{t(a_0,b_0;x)} \d \mu(x)} \text{ and } u(c;b) = \frac{t(a_0,b;c)}{\displaystyle \int_E t(a_0,x;c) \d \mu(x)}
\end{displaymath}
where $K_a$ is a normalization constant such that $\displaystyle \int_E d(a;c) \d \mu(c) = 1$. Then,
\begin{displaymath}
\truc{du}(a;b) = K_a \int_E \frac{t(a,b_0;c) t(a_0,b;c)}{t(a_0,b_0;c)} \frac{1}{\displaystyle\int_E \frac{t(a,b_0;x)}{t(a_0,b_0;x)} \d \mu(x)} \d \mu(c),
\end{displaymath}
and
\begin{eqnarray}
\frac{d(a;c) u(c;b)}{\truc{du}(a;b)} &=& \frac{t(a,b_0;c) t(a_0,b;c)}{t(a_0,b_0;c)} \frac{1}{\displaystyle \int_E \frac{t(a,b_0;x) t(a_0,b;x)}{t(a_0,b_0;x)} \d \mu(x)} \label{eq:line1} \\
& = & \frac{t(a,b_0;c) t(a_0,b;c)}{t(a_0,b_0;c)} \frac{1}{\displaystyle \int_E t(a,b;x) \frac{t(a,b_0;c_0) t(a_0,b;c_0)}{t(a_0,b_0;c_0) t(a,b;c_0)} \d \mu(x)} \label{eq:line2} \\
& = & \frac{t(a,b_0;c) t(a_0,b;c) t(a_0,b_0;c_0) t(a,b;c_0)}{t(a_0,b_0;c) t(a,b_0;c_0) t(a_0,b;c_0)} \frac{1}{\displaystyle \int_E t(a,b;x) \d \mu(x)} \label{eq:line3} \\
& = & t(a,b;c) \label{eq:line4}.
\end{eqnarray}
In this previous computation, we pass from line~(\ref{eq:line1}) to line~(\ref{eq:line2}) and from line~(\ref{eq:line3}) to line~(\ref{eq:line4}) by using \C~\ref{cond:Belyaev}. \qedhere
\end{itemize}
\end{proof}

Lemma~\ref{lem:first} says that \C~\ref{cond:Toom} is equivalent to \C~\ref{cond:Belyaev} for $\mu$-positive PCA. Next Lemma~\ref{lem:second}, gives some necessary conditions for a $(\rho_0,D,U)$-HZMC to be invariant by a $\mu$-positive PCA.
\begin{lemme} \label{lem:second}
Let $\PCA{A}$ be a $\mu$-positive PCA. If $\PCA{A}$ satisfies the conditions of Lemma~\ref{lem:first}, then there exists $H$ a $\mu$-positive probability distribution on $(E,\borel{E})$ of $\mu$-density $\eta$ such that the respective $\mu$-densities of $D$ and $U$ are, for $\mu^3$-almost $(a,b,c)$, $d^\eta$ and $u^\eta$ as defined in Eq~(\ref{eq:DUeta}).
\end{lemme}

\begin{proof}
Suppose that, for $\mu^3$ almost $(a,b,c)$, 
\begin{equation} \label{eq:base}
\truc{du}(a;b) = \frac{d(a;c) u(c;b)}{t(a,b;c)} = \frac{d(a;c_0) u(c_0;b)}{t(a,b;c_0)}.
\end{equation}
Then $\displaystyle d(a;c) u(c;b) = d(a;c_0) \frac{u(c_0;b)}{t(a,b;c_0)} t(a,b;c)$. Integrating with respect to $b$,
\begin{displaymath}
d(a;c) = d(a;c_0) \int_E \frac{u(c_0;b)}{t(a,b;c_0)} t(a,b;c) \d \mu(b)
\end{displaymath}
and, then,
\begin{equation} \label{eq:ubase}
u(c;b) = \frac{\displaystyle \frac{u(c_0;b)}{t(a,b;c_0)} t(a,b;c)}{\displaystyle \int_E \frac{u(c_0;x)}{t(a,x;c_0)} t(a,x;c) \d \mu(x)}.
\end{equation}
\C~\ref{cond:Belyaev} and \C~\ref{cond:BelyaevGen} permit to replace $a$ by $a_0$ on the right side of Eq~(\ref{eq:ubase}). Then, taking $\eta(b) = u(c_0;b)$ ends the proof.
\end{proof}

Now, we can end the proof of Theorem~\ref{theo:main}.

\begin{proof}[Proof of Theorem~\ref{theo:main}]
Let $\PCA{A}$ be a $\mu$-positive PCA. If $(\rho_0,D,U)$ is an invariant HZMC for $\PCA{A}$, then there exists $\eta \in L^1(\mu)$ such that Eq~(\ref{eq:DUeta}) holds by \C~\ref{cond:Toom}, Lemma~\ref{lem:first} and Lemma~\ref{lem:second}.

Moreover, $u$ and $d$ satisfy \C~\ref{cond:commut}. Hence, rewriting $\truc{du}$ and $\truc{ud}$ in terms of $\eta$,
\begin{eqnarray}
\truc{du}(a;b) & = & \int_E \frac{\displaystyle \int_E \frac{\eta(x)}{t(a,x;c_0)} t(a,x;c) \d \mu(x)}{\displaystyle \int_E \frac{\eta(x)}{t(a,x;c_0)} \d \mu(x)} 
\frac{\displaystyle \frac{\eta(b)}{t(a_0,b;c_0)}t(a_0,b;c)}{\displaystyle \int_E \frac{\eta(x)}{t(a_0,x;c_0)} t(a_0,x;c) \d \mu(x)} \d \mu(c) \label{eq:lineCom1}\\
& = & \int_E \frac{\displaystyle \int_E \frac{\eta(x)}{t(a,x;c_0)} t(a,b;x) \d \mu(x)}{\displaystyle \int_E \frac{\eta(x)}{t(a,x;c_0)} \d \mu(x)} 
\frac{\displaystyle \frac{\eta(b)}{t(a,b;c_0)}t(a,b;c)}{\displaystyle \int_E \frac{\eta(x)}{t(a,x;c_0)} t(a,x;c) \d \mu(x)} \d \mu(c)  \label{eq:lineCom2}\\
& = & \frac{\displaystyle \frac{\eta(b)}{t(a,b;c_0)}}{\displaystyle \int_E \frac{\eta(x)}{t(a,x;c_0)} \d \mu(x)}, \nonumber
\end{eqnarray}
we pass from line~(\ref{eq:lineCom1}) to line~(\ref{eq:lineCom2}) replacing $\displaystyle \frac{t(a_0,b;c) t(a_0,x;c_0)}{t(a_0,b;c_0) t(a_0,x;c)}$ by $\displaystyle \frac{t(a,b;c) t(a,x;c_0)}{t(a,b;c_0) t(a,x;c)}$ using \C~\ref{cond:Belyaev} and~\ref{cond:BelyaevGen}; and
\begin{eqnarray*}
\truc{ud}(a;b) & = & \int_E \frac{\displaystyle \frac{\eta(c)}{t(a_0,c;c_0)}t(a_0,c;a)}{\displaystyle \int_E \frac{\eta(x)}{t(a_0,x;c_0)} t(a_0,x;a) \d \mu(x)} 
\frac{\displaystyle \int_E \frac{\eta(x)}{t(c,x;c_0)} t(c,x;b) \d \mu(x)}{\displaystyle \int_E \frac{\eta(x)}{t(c,x;c_0)} \d \mu(x)} \d \mu(c).
\end{eqnarray*}

Hence, $\eta$ satisfies Eq~(\ref{eq:pfff}) of \C~\ref{cond:pfff}.

Finally, we need a distribution $\rho_0$ to satisfy \C~\ref{cond:invD} with $D = D^\eta$, this is possible only if \C~\ref{cond:invDex} holds.

Conversely, if we suppose \C~\ref{cond:Belyaev}, \C~\ref{cond:pfff} and \C~\ref{cond:invDex}, then all the previous computations hold and then we obtain \C~\ref{cond:Toom}, \C~\ref{cond:commut} and \C~\ref{cond:invD} for $D=D^\eta$, $U=U^\eta$ and $\rho_0$ and then we conclude using Theorem~\ref{theo:Toom}.
\end{proof}

\subsection{Proof of Proposition~\ref{prop:diag}}
Let $\PCA{A}$ be a PCA and suppose that \C~\ref{cond:Belyaev}, \C~\ref{cond:BelyaevDiag} and \C~\ref{cond:pfffDiag} hold. Then, we can replace in \C~\ref{cond:pfff} the $a_0$ by $c$ using \C~\ref{cond:Belyaev} and \C~\ref{cond:BelyaevDiag}. Then $\eta$ must verify: for $\mu$-almost $a$ and for the $c_0$ of \C~\ref{cond:Belyaev},
\begin{displaymath}
\frac{ \displaystyle \frac{\eta(a)}{t(a,a;c_0)} }{ \displaystyle \int_E \frac{\eta(x)}{t(a,x;c_0)} \d \mu(x) } = \int_E \frac{\displaystyle \frac{\eta(c)}{t(c,c;c_0)}}{\displaystyle \int_E \frac{\eta(x)}{t(c,x;c_0)} \d \mu(x)} t(c,c;a) \d \mu(c).
\end{displaymath}

So, we see that $\left(a \mapsto \frac{ \displaystyle \frac{\eta(a)}{t(a,a;c_0)} }{ \displaystyle \int_E \frac{\eta(x)}{t(a,x;c_0)} \d \mu(x) } \right)$ is an eigenfunction of the operator $\mathcal{A}_1 : f \mapsto \left( \mathcal{A}_1(f) : a \mapsto \displaystyle \int_E f(c) t(c,c;a) \d \mu(c) \right)$. Hence, by Lemma~\ref{lem:Durrett}, if there exists a positive eigenfunction $\nu$ in $L^1(\mu)$ for $\mathcal{A}_1$, it is unique up to a multiplicative constant. Hence, there exists $\lambda>0$ such that, for $\mu$-almost $a$,
\begin{displaymath}
\frac{ \displaystyle \frac{\eta(a)}{t(a,a;c_0)} }{ \displaystyle \int_E \frac{\eta(x)}{t(a,x;c_0)} \d \mu(x) } = \lambda \nu(a),
\end{displaymath}
which is equivalent to
\begin{displaymath}
\eta(a) = \lambda \int_E \eta(x) \frac{t(a,a;c_0)}{t(a,x;c_0)} \nu(a) \d \mu(x).
\end{displaymath}

Hence, $\eta$ is an eigenfunction of $\mathcal{A}_2 : f \mapsto \left( \mathcal{A}_2(f) : a \mapsto \displaystyle \int_E f(x) \frac{t(a,a;c_0)}{t(a,x;c_0)} \nu(a) \d \mu(x) \right)$.

\section{Extension to $\mathbb{Z}$ and $\cyl{n}$} \label{sec:extension}
\subsection{PCA on $\mathbb{Z}$}\label{sec:PCAonZ}
In this section, we extend Theorem~\ref{theo:Toom} and~\ref{theo:main} to $\mathbb{Z}$. The main change is that $\rho_0$ the initial probability distribution for a HZMC on $\mathbb{N}$ is replaced on $\mathbb{Z}$ by a sequence of probability distributions $R=(\rho_i)_{i \in \mathbb{Z}}$ indexed by $\mathbb{Z}$.

Let us define a $\text{HZMC}_{\mathbb{Z}}$ on $\mathbb{Z}$. The geometrical structure is now
\begin{displaymath}
\HZ_{\mathbb{Z}}(t) = \left\{ \left( \left\lfloor \frac{i}{2} \right\rfloor, t + \frac{1+(-1)^{i+1}}{2} \right), i \in \mathbb{Z} \right\}.
\end{displaymath}
See Figure~\ref{fig:HZMC_Z} for a graphical representation.\par
On this structure, a $(R,D,U)$-$\text{HZMC}_{\mathbb{Z}}$ is a Markov chain with two t.k. $D$ and $U$ and a family of probability distributions $R = \left( \rho_i \right)_{i \in \mathbb{Z}}$ such that
\begin{itemize}
  \setlength{\itemsep}{1pt}
  \setlength{\parskip}{0pt}
  \setlength{\parsep}{0pt}
\item for all $i \in \mathbb{Z}$, the distribution of state $S(i,t)$ is $\rho_i$,
\item the distribution of $S(i,t+1)$ knowing $S(i,t)$ is given by $D$, and that of $S(i+1,t)$ knowing $S(i,t+1)$ is given by $U$.
\end{itemize}\par

Hence, for any $i \in \mathbb{Z}$, the distributions $\rho_{i}$, $\rho_{i+1}$, $D$ and $U$ are constrained such that $\rho_i DU = \rho_{i+1}$. In the case of a $\mu$-supported $(R,D,U)$-$\text{HZMC}_{\mathbb{Z}}$ (i.e. for all $i \in \mathbb{Z}$, $\rho_i \ll \mu$ and $D$ and $U$ are $\mu$-supported t.k.), this gives
\begin{equation}\label{eq:compa}
r_{i+1}(x_{i+1}) = \int_{E} r_{i}(x_i) \truc{du}(x_i;x_{i+1}) \d \mu(x_i).
\end{equation}
A family of probability distributions $R$ that possesses this property is said to be compatible with $(D,U)$.\par

As before, we define the support $\tilde{E} = \bigcup_{i \in \mathbb{Z}} \text{supp}\left(\rho_i\right)$ of an $\text{HZMC}_{\mathbb{Z}}$. 
If the $\text{HZMC}_{\mathbb{Z}}$ is $\mu$-supported, then, for $\mu$-almost $a \in E$, there exists $i$ such that $r_i(a) > 0$ and,
in the case of a $\mu$-positive $\text{HZMC}_{\mathbb{Z}}$ (i.e. for all $i \in \mathbb{Z}$, $\rho_i$ and $\mu$ are positive equivalent and $D$ and $U$ are $\mu$-positive t.k.), $\tilde{E} = \text{supp}(\mu)$.

\begin{figure}
\begin{center}
\includegraphics{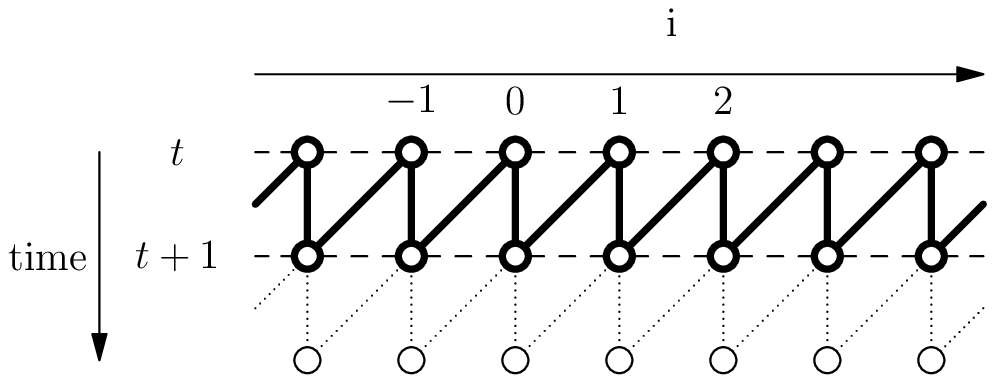}
\caption{In bold, $\HZ_{\mathbb{Z}}(t)$, the $t$th horizontal zigzag on $\mathbb{Z}$ on a space time diagram.}
\label{fig:HZMC_Z}
\end{center}
\end{figure}

The following theorem is an extension of Theorem~\ref{theo:Toom} for PCA on $\mathbb{Z}$.
\begin{theoreme} \label{theo:ToomZ}
Let $\mu$ be a $\sigma$-finite measure on $E$. Let $\PCA{A} := (\mathbb{Z}, E, N, T)$ be a $\mu$-supported PCA and $(R,D,U)$ a $\mu$-supported $\text{HZMC}_{\mathbb{Z}}$.
The $(R,D,U)$-$\text{HZMC}_{\mathbb{Z}}$ is invariant by $\PCA{A}$ if and only if \C~\ref{cond:Toom}, \C~\ref{cond:commut} and the following condition are satisfied:
\begin{cond}\label{cond:invDZ}
$\rho$ is invariant by t.k. $D$, i.e. for all $i \in \mathbb{Z}$, $\rho_i D = \rho_i$, i.e. for all $i \in \mathbb{Z}$, $\displaystyle r_i(c) = \int_E r_i(a) d(a;c) \d \mu(a)$.
\end{cond}
\end{theoreme}

\begin{proof}
This is an immediate consequence of Theorem~\ref{theo:Toom} because we just need, for all $i \in \mathbb{Z}$, the $(\rho_i,D,U)$-HZMC to be invariant by $\PCA{A}$.
\end{proof}

As in Theorem~\ref{theo:main} where we go further for $\mu$-positive PCA on $\mathbb{N}$, we obtain a necessary and sufficient condition on $\mu$-positive PCA to have an invariant $\text{HZMC}_{\mathbb{Z}}$.
\begin{theoreme} \label{theo:mainZ}
Let $\mu$ be a $\sigma$-finite measure on $E$. Let $\PCA{A} := (\mathbb{Z}, E, N, T)$ be a $\mu$-positive PCA. $\PCA{A}$ admits a $\mu$-positive invariant $\text{HZMC}_{\mathbb{Z}}$ if and only if \C~\ref{cond:Belyaev}, \C~\ref{cond:pfff} and \C~\ref{cond:invDex} hold.

In this case, the $(R,D,U)$-$\text{HZMC}_{\mathbb{Z}}$ has for $\mu$-densities $d^\eta$ and $u^\eta$ as defined in Eq~(\ref{eq:DUeta}) and, for any $i \in \mathbb{Z}$, $\rho_i=\rho_0$.
\end{theoreme}

\begin{proof}
It is an immediate consequence of Theorem~\ref{theo:main}. The only new thing to verify is that $R=\rho_0^{\mathbb{Z}}$ is $(D,U)$ compatible, i.e. $r_0$ satisfies Eq~(\ref{eq:compa}) to check that $(R,D,U)$ defines a $\text{HZMC}_{\mathbb{Z}}$. By \C~\ref{cond:invDex}, for $\mu$-almost $y_i$, 
\begin{equation} \label{eq:invDproof}
\int_E r_0(x_i) d(x_i;y_i) \d \mu(x_i) = r_0(y_i). 
\end{equation}
But, satisfying Eq~(\ref{eq:compa}) and Eq~(\ref{eq:invDproof}) is equivalent to satisfy, for $\mu$-almost $x_{i+1}$,
\begin{displaymath}
\int_E r_0(y_i) u(y_i;x_{i+1}) \d \mu(y_i) = r_0(x_{i+1}). 
\end{displaymath}
Now, from Eq~(\ref{eq:invDproof}), for $\mu$-almost $x_{i+1}$,
\begin{displaymath}
\iint_{E^2} r_0(x_i) d(x_i;y_i) u(y_i;x_{i+1}) \d \mu(x_i) \d \mu(y_i) = \int_E r_0(y_i) u(y_i;x_{i+1}) \d \mu(y_i).
\end{displaymath}
But as $\truc{du} = \truc{ud}$,
\begin{displaymath}
  \int_E \left(\int_E r_0(x_i) u(x_i;y_i) \d \mu(x_i)\right) d(y_i;x_{i+1}) \d \mu(y_i) = \int_E r_0(y_i) u(y_i;x_{i+1}) \d \mu(y_i).
\end{displaymath}
So, $f: y \rightarrow \displaystyle \int_E r_0(x) u(x;y) \d \mu(x)$ is a positive eigenfunction of the integral operator $\mathcal{A}$ of kernel $d$. By Lemma~\ref{lem:Durrett}, this eigenfunction is unique (up to a multiplicative constant) equal to $r_0$, so $\displaystyle \int_E r_0(x) u(x;y) \d \mu(x) = \lambda r_0(y)$ and $\lambda = 1$ because they both integrates (with respect to $\mu$) to $1$. This ends the proof.
\end{proof}

In that case, Prop~\ref{prop:diag} still holds and Prop~\ref{prop:sameInv} also holds if $(\rho_0,D,U)$-HZMC is replaced by $(R,D,U)$-$\text{HZMC}_{\mathbb{Z}}$.

\subsection{PCA on $\cyl{n}$}\label{sec:PCAonC}
In this section, we have results, similar to Theorems~\ref{theo:Toom} and~\ref{theo:main}, on the lattice $\cyl{n}$. The main change is that we characterize PCA whose invariant distribution is a cyclic-HZMC (CHZMC).\par
Consider, as represented on Figure~\ref{fig:HZMC_C},
\begin{displaymath}
\CHZ(t) = \left\{ \left( \left\lfloor \frac{i}{2} \right\rfloor, t + \frac{1+(-1)^{i+1}}{2} \right), i \in \cyl{(2n)} \right\}.
\end{displaymath}

\begin{figure}
\begin{center}\includegraphics{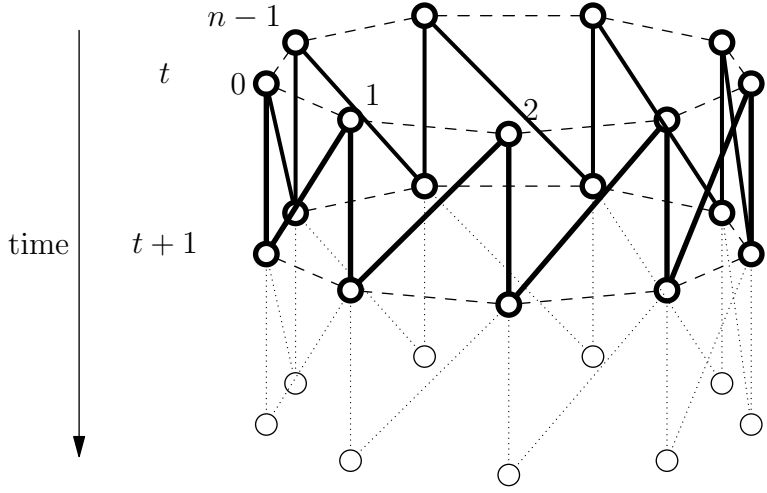}
\caption{In bold, $\CHZ(t)$, the $t$th cyclic horizontal zigzag on a space time diagram.}
\label{fig:HZMC_C}
\end{center}
\end{figure}

Let $(D,U)$ be two $\mu$-supported t.k. from $E$ to $E$ such that
\begin{displaymath}
Z(D,U) = \int_{E^{2n}} u(y_{n-1};x_0) d(x_0;y_0) \dots d(x_{n-1};y_{n-1}) \d \mu^{2n}(x_0,y_0,x_1,\dots,y_{n-1}) \notin \{0,+\infty\}.
\end{displaymath}
We define the measure $M$ on $\CHZ$ called ($\mu$-supported) $(D,U)$-CHZMC by its $\mu^{2n}$-density $m$ that is: for $\mu$-almost $x_0,y_0,\dots,y_{n-1} \in E$,
\begin{displaymath}
m(x_0,y_0,\dots,y_{n-1}) = \frac{u(y_{n-1};x_0) d(x_0;y_0) \dots d(x_{n-1};y_{n-1})}{Z(D,U)}.
\end{displaymath}\par

For simplicity, we define, formally, only $\mu$-supported $(D,U)$-CHZMC ($D$ and $U$ are $\mu$-supported t.k. from $E$ to $E$).\par

When $E$ is finite, a CHZMC is a HZMC conditioned to be periodic. In general, a CHZMC is a Gibbs measure in the cyclic horizontal zigzag (CHZ).\par

Cyclic Markov chain were introduced, first, by Albenque~\cite{Alb09} to define periodic Markov chain on $\cyl{n}$.\par

\begin{notation}
The distribution of the line $S_t$ (resp. $S_{t+1}$) is denoted $M^{(1)}$ (resp. $M^{(2)}$) and its $\mu^n$-density is obtained by integration of $m$ with respect to the $n$ variables $y_0,\dots,y_{n-1}$ (resp. to the $n$ variables $x_0,\dots,x_{n-1}$). The distribution of the state $S(i,t)$ is denoted $M^{(1)}_i$ and its $\mu$-density is obtained by integration of $m$ with respect to the $2n-1$ variables $x_0,y_0,\dots,x_{i-1},y_{i-1},y_{i},x_{i+1},\dots,x_{n-1}$.

For any $j \in \mathbb{N}$, for $\mu$-almost $a,b$, we let
\begin{displaymath}
\truc{(du)^j}(a;b) = \int_{E^{2j-1}} d(a;y_0) u(y_0;x_1) \dots u(y_{j-1};b) \d \mu^{2j-1}(y_0,x_1,\dots,y_{j-1}).
\end{displaymath}
\end{notation}\par

We obtain, first, a theorem about $\mu$-supported PCA having $\mu$-supported CHZMC. 
\begin{theoreme}\label{theo:ToomC}
Let $\mu$ be a $\sigma$-finite measure on $E$. Let $\PCA{A} := (\cyl{n}, E, N, T)$ be a $\mu$-supported PCA and $(D,U)$ a $\mu$-supported CHZMC.
The $(D,U)$-CHZMC is invariant by $\PCA{A}$ if and only if the two following conditions are satisfied:

\begin{cond}\label{cond:ToomC}
for $\mu$-almost $a,b,c \in E$,
\begin{displaymath}
\truc{du}(a;b) t(a,b;c) = d(a;c) u(c;b) \text{ or } \truc{(du)^{n-1}}(b;a) = 0,
\end{displaymath}
\end{cond}

\begin{cond}\label{cond:commutC}
for $\mu$-almost $x_0, x_1, \dots,x_{n-1} \in \tilde{E}$,
\begin{displaymath}
\truc{du}(x_0;x_1) \truc{du}(x_1;x_2) \dots \truc{du}(x_{n-1};x_0) = \truc{ud}(x_0;x_1) \truc{ud}(x_1;x_2) \dots \truc{ud}(x_{n-1};x_0) 
\end{displaymath}
\end{cond}
\end{theoreme}

\begin{proof}
Let a $(D,U)$-CHZMC be invariant by $\PCA{A}$. For all $A,B,C \in \borel{E}$, for all $i \in \cyl{n}$,
\begin{eqnarray*}
&&\prob{S(i,t) \in A, S(i+1,t) \in B, S(i,t+1) \in C} \\
&& = \frac{1}{Z(D,U)} \int_{A \times C \times B} d(x_i;y_i) u(y_i;x_{i+1}) \truc{(du)^{n-1}}(x_{i+1};x_i) \d \mu^3(x_i,y_i,x_{i+1}) \\
&& = \frac{1}{Z(D,U)} \int_{A \times C \times B} \truc{du}(x_i;x_{i+1}) t(x_i,x_{i+1};y_i) \truc{(du)^{n-1}}(x_{i+1};x_i) \d \mu^3(x_i,y_i,x_{i+1}).
\end{eqnarray*}
Hence, for $\mu$-almost $x_i,y_i,x_{i+1} \in E$, 
\begin{displaymath}
\truc{du}(x_i;x_{i+1}) t(x_i,x_{i+1};y_i) \truc{(du)^{n-1}}(x_{i+1};x_i) = d(x_i;y_i) u(y_i;x_{i+1}) \truc{(du)^{n-1}}(x_{i+1};x_i),
\end{displaymath}
that is \C~\ref{cond:ToomC}.

To prove \C~\ref{cond:commutC}, we use the fact that the second line of the $(D,U)$-CHZMC at time $t$ is the first line at time $t+1$ and since the CHZMC is invariant the law of the CHZMC at time $t$ and at time $t+1$ is the same $M$. But $M^{(1)}$ is the law of the first line and $M^{(2)}$ of the second, so $M^{(1)} = M^{(2)}$. In terms of $\mu^n$-densities, $m^{(1)}=m^{(2)}$. But, $\displaystyle m^{(1)}(x_0,\dots,x_{n-1}) = \frac{1}{Z(D,U)} \truc{du}(x_0;x_1) \dots \truc{du}(x_{n-1};x_0)$ and $\displaystyle m^{(2)}(x_0,\dots,x_{n-1}) = \frac{1}{Z(D,U)} \truc{ud}(x_0;x_1) \dots \truc{ud}(x_{n-1};x_0)$ that gives \C~\ref{cond:commutC}.

Conversely, we suppose that \C~\ref{cond:ToomC} and~\ref{cond:commutC} are satisfied. Then, the push forward measure of the $(D,U)$-CHZMC by $\PCA{A}$ is also the $(D,U)$-CHZMC (the computation is an adaptation of that done in the Proof of Theorem~\ref{theo:Toom} in Section~\ref{sec:proofToom} to compute the push forward measure of a HZMC).
\end{proof}

For $\mu$-positive PCA, \C~\ref{cond:ToomC} could be exploited a little more.
\begin{theoreme}\label{theo:mainC}
Let $\mu$ be a $\sigma$-finite measure on $E$. Let $\PCA{A} := (\cyl{n}, E, N, T)$ be a $\mu$-positive PCA. $\PCA{A}$ admits a $\mu$-positive invariant CHZMC if and only if \C~\ref{cond:Belyaev} and

\begin{cond}\label{cond:pfffC}%
there exists a positive function $\eta \in L^1(\mu)$ solution of: for $\mu$-almost  $x_0,\dots,x_{n-1} \in E$,
\begin{eqnarray*}
\truc{d^\eta u^\eta }(x_0;x_1) \truc{d^\eta u^\eta }(x_1;x_2) \dots \truc{d^\eta u^\eta }(x_{n-1};x_0) = \truc{u^\eta d^\eta }(x_0;x_1) \truc{u^\eta d^\eta }(x_1;x_2) \dots \truc{u^\eta d^\eta }(x_{n-1};x_0)
\end{eqnarray*}
with $d^\eta$ and $u^\eta$ as defined in Eq~(\ref{eq:DUeta}).

In this case, the $(D,U)$-CHZMC has for $\mu$-densities $d^\eta$ and $u^\eta$ as defined in Eq~(\ref{eq:DUeta}).
\end{cond}
\end{theoreme}

\begin{proof}
First of all, when a PCA is $\mu$-positive, \C~\ref{cond:ToomC} can be rewritten, for $\mu$-almost $a,b,c$, $t(a,b;c) = \displaystyle \frac{d(a;c) u(c;b)}{\truc{du}(a;b)}$ because both $\truc{du}(a;b)$ and $\truc{(du)^{n-1}(b;a)}$ are positive. Hence, we use Lemma~\ref{lem:first} to prove that \C~\ref{cond:ToomC} is equivalent to \C~\ref{cond:Belyaev}. Moreover, Lemma~\ref{lem:second} still applies and the state space of possible solutions for $(D,U)$ is parametrized by $\eta$ a function in $L^1(\mu)$. \C~\ref{cond:commutC} applied on $d^\eta$ and $u^\eta$ gives \C~\ref{cond:pfffC}.
\end{proof}

\bibliography{biblio}

\section*{Acknowledgments}
I am very grateful to Jean-Fran\c cois Marckert to encourage me to work on this topic. His comments and suggestions have been a great benefit.
\end{document}